\documentclass[final]{siamltex}
\usepackage{amsfonts}
\usepackage{calrsfs}
\usepackage{makeidx}
\usepackage{amsmath}
\usepackage{amssymb}
\usepackage{bm}
\usepackage{cite}
\usepackage{showkeys}
\usepackage[english]{babel}
\usepackage{dblaccnt}
\usepackage{accents}
\usepackage{graphicx}
\usepackage{psfrag}
\usepackage{pstool}
\usepackage{subfig}
\usepackage{color}
\usepackage{amscd}
\usepackage[mathscr]{euscript}
\newcommand{\N}{\ensuremath{\mathbb{N}}}

\renewcommand{\phi}{\varphi}
\renewcommand{\d}{\mathrm{d}}

\newcommand{\ol}{\overline}

\renewcommand{\Re}{\mathrm{Re}\,}
\renewcommand{\Im}{\mathrm{Im}\,}

\newcommand{\om}{\omega}
\newcommand{\eps}{\varepsilon}

\newcommand{\vbJ}{{\vec{\bf J}}}

\newcommand{\curl}{{\vec{\nabla}\times}}
\newcommand{\divv}{{\vec{\nabla}\cdot}}
\renewcommand{\div}{\mathrm{div}}
\newcommand{\curll}{\mathrm{curl}}

\newcommand{\pa}{\partial}

\newcommand{\vbx}{\vec{\bm{x}}}
\newcommand{\vby}{\vec{\bm{y}}}

\newcommand{\vx}{\vbx}
\newcommand{\bx}{\bm{x}}
\newcommand{\vbn}{\vec{\bm{n}}}

\newcommand{\vbu}{\vec{\bm{u}}}
\newcommand{\vbg}{\vec{\bm{g}}}
\newcommand{\vbf}{\vec{\bm{f}}}

\newcommand{\vbv}{\vec{\bm{v}}}
\newcommand{\vbphi}{\vec{\bm{\varphi}}}

\newcommand{\cF}{{\mathcal F}}

\newcommand{\lin}{\big<}
\newcommand{\rin}{\big>}
\newcommand{\bV}{\boldsymbol{\mathnormal{V}}}
\newcommand{\bF}{{\bf F}}
\newcommand{\bd}{{\bf d}}
\newcommand{\cL}{{\mathcal L}}
\newcommand{\icu}{{\mbox{\tiny curl}}}
\newcommand{\idiv}{{\mbox{\tiny div}}}

\begin{document}

\sloppy

\title{Imaging with electromagnetic waves in terminating waveguides}
\author{Liliana Borcea
  \and Dinh-Liem Nguyen\thanks{Department of Mathematics, University
    of Michigan, Ann Arbor, MI, 48109, USA;
    \texttt{borcea@umich.edu and dlnguyen@umich.edu}} }

\maketitle

\begin{abstract}
  We study an inverse scattering problem for Maxwell's equations in
  terminating waveguides, where localized reflectors are to be imaged
  using a remote array of sensors. The array probes the waveguide with
  waves and measures the scattered returns.  The mathematical
  formulation of the inverse scattering problem is based on the
  electromagnetic Lippmann-Schwinger integral equation and an explicit
  calculation of the Green tensor.  The image formation is carried
  with reverse time migration and with $\ell_1$ optimization.
\end{abstract}


\section{Introduction}
We consider an inverse scattering problem for Maxwell's equations in a
waveguide which contains a few unknown reflectors. The setup is
illustrated in Figure \ref{fig:setup}, where an array of sensors
probes the waveguide with waves and records the returns over the
duration of some time window. The inverse problem is to reconstruct
the reflectors from these measurements.

To carry out explicit calculations we assume that the waveguide has a
simple geometry, with rectangular cross-section $\Omega =
(0,L_1)\times(0,L_2)$, and introduce the system of coordinates $\vx =
(\bx,x_3)$, with $\bx = (x_1,x_2) \in \Omega$ and $x_3 \le 0$. The
waveguide terminates at $x_3 = 0$ and we denote its domain by
\[
W = (0,L_1)\times(0,L_2)\times(-\infty,0), 
\]
with boundary $\pa W$. For convenience we model the boundaries as
perfectly conducting, but other boundary conditions may be used.

The electric field $\vec{\bm E}(\om,\vx)$, decomposed over frequencies $\om$,
satisfies the equation
\begin{align}
\curl \curl \vec{\bm E}(\om,\vbx) - \om^2 \mu_o \eps(\om,\vbx) \vec{\bm E}(\om,\vbx)
= i \om \mu_o \vbJ(\om, \bm x) \delta(x_3 + L), \label{eq:Maxwell1}
\qquad \vx \in W,
\end{align}
with boundary conditions
\begin{align}
\label{eq:boundaryCond}
 \vbn(\vx)\times\vec{\bm E}(\omega,\vbx) = 0  \quad \text{on } \pa W,
\end{align}
where $\curl$ is the curl operator in $\mathbb{R}^3$ and $\vbn(\vx)$
is the unit outer normal at $\pa W$. There is also a radiation
condition at $x_3 \to -\infty$, which states that $\vec{\bm E}(\om,\vx)$ is
bounded and outgoing.  The current source density $\vbJ$ models the
excitation from the array located at distance $L$ from the
terminating boundary at $x_3 = 0$.

The waveguide is filled with a linear and isotropic homogeneous medium
with electric permittivity $\eps_o$ and magnetic permeability $\mu_o$,
and a few reflectors supported in the compact domain $D \subset W$,
located between the array and the terminating boundary. The reflectors
are modeled as linear and possibly anisotropic dielectrics with
Hermitian, positive definite relative electric permitivity matrix
$\eps_r(\om,\vbx)$.  The term $\eps \vec{\bm E}$ in (\ref{eq:Maxwell1}) is
the electric displacement, and $\eps$ is the electric permittivity
tensor satisfying
\begin{equation}
  \eps(\om,\vx) = \eps_o \Big[ 1_{_D}(\vx) \Big(\eps_r(\om,\vx) - I) + 
      I \Big].
  \label{eq:3}
\end{equation}
Here $I$ is the $3\times 3$ identity matrix and $1_{_D}(\vx)$ is the
indicator function, equal to one for $\vx \in D$ and zero otherwise.

\begin{figure}[t!]
\centering
   \psfrag{x1}{$x_1$}
   \psfrag{x2}{$x_2$}
   \psfrag{x3}{$x_3$}
   \psfrag{D}{$D$}
   \psfrag{L1}{$L_1$}
   \psfrag{L2}{$L_2$}
   \psfrag{L}{$L$}
    \psfrag{O}{$O$}
   \includegraphics[width=12cm]{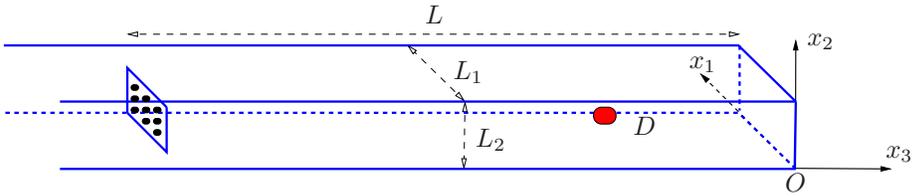}
   \caption{Schematic of the imaging setup in a terminating waveguide
     with rectangular cross-section. The unknown reflector is
     supported in $D$. The array of sensors is far away from it, at
     distance $L$ from the terminating boundary. }    
   \label{fig:setup}
\end{figure}

The inverse problem is to reconstruct the perturbation $\eps_r-I$ in
(\ref{eq:3}), or at least its support $D$, from measurements of the
electric field $\vec{\bm E}(\om,\vbx)$ at points $\vbx =(\bx,-L)$ in the
array aperture $A$, a subset of $\Omega$.

Inverse scattering and inverse source problems in waveguides  have been
considered in the past in various setups relevant to applications in ocean acoustics,
non-destructive evaluation and imaging  in tunnels.
We refer to
\cite{Dediu2006, Bourg2008,
  Bourg2011, Bourg2012, Bourg2013, Monk2012, Arens2011, Xu2000,
  Roux2000,tsogka2013selective} for mathematical studies of inverse scattering problems in acoustic and elastic
waveguides with straight walls, and filled with homogeneous media. Random acoustic
waveguides with finite cross-section are considered in \cite{Borce2010, Issa2010}, and
with  unbounded cross-section, as encountered in ocean acoustics, in  \cite{borcea2014paraxial,sabra2004blind}.
Examples of inverse scattering problems in planar electromagnetic waveguides are in
\cite{tamil1991spectral,Mills1992, Jorda1996},
where  the problem  is reduced to one for the  scalar Helmholtz equation by considering a single type of waves, transverse electric or magnetic.

In this paper we give the mathematical formulation of the
electromagnetic scattering problem in terminating waveguides and study
with numerical simulations two imaging methods. The first is a
reverse time migration approach, where the wave field measured at the
array is time reversed and propagated to the imaging region using the
electromagnetic Green's tensor in the unperturbed waveguide. The
second method uses $\ell_1$ optimization, and is motivated by the
assumption that the perturbation of the electric permittivity has
small spatial support $D$.

The paper is organized as follows: We begin in section \ref{sect:FP}
with the formulation of the forward problem. We define the scattered
electric field and show that it satisfies a Lipmann-Schwinger type
equation. The solvability of this equation is analyzed using the
Fredholm alternative. The data model used for inversion is given in
section \ref{sect:FP4} and the imaging methods are formulated in
section \ref{sect:imag}. The imaging results obtained with numerical
simulations are in section \ref{sect:num}. We end with a summary in
section \ref{sect:sum}.

\section{The scattering problem}
\label{sect:FP}
In this section we formulate the scattering problem. We begin in section
\ref{sect:FP1} with the expression of the electric field in the
unperturbed (homogeneous) waveguide. Then we define in section
\ref{sect:FP2} the scattered wave field at the unknown reflectors and
derive the radiation condition at $x_3 \to -\infty$. We state the
scattering problem as a Lippmann-Schwinger integral equation and prove
its Fredholm property in section \ref{sect:FP3}. 
\subsection{The homogeneous waveguide}
\label{sect:FP1}
In the absence of any reflector in the waveguide the electric field is
denoted by $\vec{\bm E}^o$, and solves the boundary value problem 
\begin{align}
\nonumber \curl\curl \vec{\bm{E}}^{o}(\vx) - k^2 \vec{\bm{E}}^{o}(\vx)
&= i\omega\mu_o\vbJ(\bm x)\delta(x_3+L) \qquad \vx = (\bx,x_3)\in W,
\\
\label{eq:forward4}
\vbn(\vbx) \times\vec{\bm{E}}^{o}(\vx) &= 0  \qquad \vx \in \pa W, 
\end{align}
where $k = \om \sqrt{\eps_o \mu_o}$ is the wavenumber. Obviously,
$\vec{\bm E}^o$ and $\vbJ$ depend on the frequency $\om$, but since we
consider a constant $\om$ we simplify notation and drop it henceforth
from the arguments of all fields. The expression of the electric field
in infinite homogeneous waveguides is well known. See for example
\cite[chapter 8]{jackson}. It is a superposition of a countable set of transverse
electric and magnetic waves, called modes, which are either
propagating away from the source or are decaying. In the terminating
waveguide we have a similar mode decomposition of $\vec{\bm E}^o$, as stated
in Lemma \ref{lem.1}, but there are both outgoing (forward
propagating) and incoming (backward propagating) waves due to the
reflection at the terminating boundary at $x_3 = 0$, and the
evanescent waves may be growing or decaying away from the source, in
the interval  $x_3  \in (-L,0)$.

The mode decomposition in Lemma \ref{lem.1} is obtained by expanding
at each $x_3$ the field $\vec{\bm E}^o(\vx)$ in the eigenfunctions $\vec
\Phi_n^{(s)}(\bx)$ of the vectorial Laplacian
\begin{align}
-\Delta \vec \Phi_n^{(s)}(\bx) &= \lambda_n \vec \Phi_n^{(s)}(\bx) \quad \bx \in  \Omega,
\nonumber \\ \bm{n}^\perp(\bx)\cdot \Phi_n^{(s)}(\bx) &= \Phi_{n,3}(\bx) = 0\quad \bx \in 
\pa\Omega, \nonumber \\ \nabla\cdot {\Phi}(\bx) &= 0 \quad \bx \in 
\pa\Omega. \label{eq:vectLapl}
\end{align}
We refer to appendix \ref{sect:VEP} for a proof that $\{\vec\Phi_n^{(s)}(\bm x)\}_{n \in \N^2_0, 1 \le s \le
  {m}_n}$ is an orthogonal basis of $\big(L^2(\Omega)\big)^3$, and to  \cite[section 3]{alonso2015electromagnetic} for an
explanation of why the basis is useful for the analysis of
electromagnetic waves in waveguides with perfectly conducting
boundaries.  In (\ref{eq:vectLapl}) the Laplacian $\Delta$ and
divergence $\nabla \cdot$ are with respect to $\bx \in \Omega$,
$\bm{n}$ is the outer normal at $\pa \Omega$, and $\bm{n}^\perp$ is
its rotation by $90$ degrees, counter-clockwise.  The vectors $\vec
\Phi_n^{(s)} = (\Phi_n^{(s)},\Phi_{n,3}^{(s)})$ are written in terms of their two dimensional
projection $\Phi_n^{(s)}$ in the cross-section plane and the longitudinal part
$\Phi_{n,3}^{(s)}$.  The eigenvalues $\lambda_n$ and eigenvectors $\vec
\Phi_n^{(s)}$ are indexed by $n \in \N^2_0 = \{(n_1,n_2): n_1^2+n_2^2
\neq 0\}$ and the multiplicity index $s = 1, \ldots, m_n$.

\begin{lemma}
\label{lem.1}
The solution of (\ref{eq:forward4}) has the following mode
decomposition
\begin{align}
\label{eq:Eo1}
\vec{\bm{E}}^{o}(\vbx) = \sum_{ n\in\N^2_0}\sum_{s=1}^{m_n}
\vec\Phi_n^{(s)}(\bm x) \Big(a^{+(s)}_{o,n} e^{i\beta_nx_3} +
b^{+(s)}_{o,n}e^{-i\beta_nx_3}\Big), \quad \text{for } x_3 \in (-L,0),
\end{align}
and 
\begin{align}
\vec{\bm{E}}^{o}(\vbx) = \sum_{ n\in\N^2_0}\sum_{s=1}^{m_n}
\vec\Phi_n^{(s)}(\bm x)  \, b^{-(s)}_{o,n}e^{-i\beta_nx_3}, \quad
\text{for } x_3 < -L,
\label{eq:Eo2}
\end{align}
where $a_{o,n}^+(s)$ and $b_o^{\pm(s)}$ are constant mode amplitudes
determined by the current excitation $\vbJ(\bx)$, and the superscripts
$\pm$ remind us that the field is evaluated in the forward direction
(toward the terminating boundary) or away from it.  The modes are
waves with wavenumber
\begin{align}
\beta_n= \begin{cases}
              \sqrt{k^2 - \lambda_n}, & k^2 \geq \lambda_n, \\
              i \sqrt{\lambda_n- k^2}, & k^2 < \lambda_n.
            \end{cases}
\end{align}
For a finite number of indexes $n \in \N^2_0$ the wavenumbers 
$\beta_n$ are real valued and the waves are propagating. The remaining
infinitely many waves are evanescent.
\end{lemma}
\begin{proof}Equations (\ref{eq:Eo1})-(\ref{eq:Eo2}) are obtained by solving
(\ref{eq:forward4}) with separation of variables. Since the
eigenfunctions of the vectorial Laplacian in (\ref{eq:vectLapl}) form
an orthogonal basis of $\big(L^2(\Omega)\big)^3$, as shown in
Appendix \ref{sect:VEP}, we can expand $\vec{\bm E}^o$ in this basis for each
$x_3 \ne - L$. Equations (\ref{eq:Eo1})-(\ref{eq:Eo2}) follow by
substitution in (\ref{eq:forward4}) and straightforward calculation
given in appendix \ref{sect:REF}. The mode amplitudes are derived from
jump conditions at the source coordinate $x_3 = -L$, reflection
conditions at the terminating boundary at $x_3=0$, and the radiation
condition at $x_3 \to -\infty$. The boundary conditions at $\pa
\Omega$ are built into the expansion in the basis $\{\vec
\Phi_n^{(s)}\}$. Note that in the interval $x_3 \in (-L,0)$ between
the source and the terminating boundary there are both forward and
backward propagating waves and decaying and growing evanescent
waves. On the other side of the source, for $x_3 < -L$, the
propagating waves are outgoing and the evanescent waves are decaying,
as imposed by the radiation condition.
\end{proof}

The simple geometry of the waveguide, with rectangular cross-section,
allows us to write explicitly the mode decomposition in
(\ref{eq:Eo1})--(\ref{eq:Eo2}). The eigenvalues are 
\begin{equation}
  \lambda_n = \left(\frac{\pi n_1}{L_1}\right)^2 + \left(\frac{\pi
    n_2}{L_2}\right)^2, \qquad n = (n_1,n_2)\in \N_o^2,
\label{eq:eigenvals}
\end{equation}
and by assuming that $(L_1/L_2)^2$ is not a rational number, we ensure that 
$\lambda_n \ne \lambda_{n'}$ if $n = (n_1,n_2) \ne n' =
(n_1',n_2')$. This limits the multiplicity $m_n$ of the eigenvalues to
\begin{equation}
m_n = \left\{ \begin{array}{ll} 
1  &\mbox{if}~ ~ n_1 n_2 = 0, \\
3 &\mbox{otherwise}. \end{array} \right.
\label{eq:multiplic}
\end{equation}
For the index pairs satisfying $n_1 n_2 = 0$, the eigenvalues are
simple, with eigenvectors 
\begin{equation}
\vec \Phi_n^{(1)}(\bx) = \delta_{n_20}\left( \begin{matrix} 0
  \\ \\\sin \big( \frac{\pi n_1x_1}{L_1}\big) \\ 0 \end{matrix}
\right) + \delta_{n_10} \left( \begin{matrix} \sin \big(\frac{\pi
    n_2x_2}{L_2}\big) \\ 0 \\ 0 \end{matrix}
\right), \label{eq:eigvect1} 
\end{equation}
satisfying the divergence free condition $\vec{\nabla} \cdot \vec
\Phi^{(1)}(\vx) = 0$.  Otherwise, there is triple multiplicity of the
eigenvalues, and the eigenvectors are given by
\begin{align}
\vec\Phi^{(1)}_n(\bx) &= 
\left( \begin{matrix} \frac{\pi n_2}{L_2} \cos \big(\frac{\pi
    n_1x_1}{L_1}\big) \sin \big(\frac{\pi n_2x_2}{L_2}\big)
  \\ -\frac{\pi n_1}{L_1}\sin \big( \frac{\pi n_1x_1}{L_1}\big) \cos
  \big( \frac{\pi n_2x_2}{L_2}\big) \\ 0 \end{matrix} \right),
\label{eq:eigvect2} \\
\vec\Phi^{(2)}_n(\bx) &=
\left( \begin{matrix} \frac{\pi n_1}{L_1} \cos \big(\frac{\pi
    n_1x_1}{L_1}\big) \sin \big(\frac{\pi n_2x_2}{L_2}\big)\\ \frac{\pi
    n_2}{L_2}\sin \big( \frac{\pi n_1x_1}{L_1}\big) \cos \big(
  \frac{\pi n_2x_2}{L_2}\big) \\ 0
\end{matrix}\right), \label{eq:eigvect3}
\end{align}
which are vectors in the cross-range plane, satisfying the divergence
free condition $\vec{\nabla} \cdot \vec \Phi^{(1)}(\vx) = 0$ and the
curl free condition $\curl \vec \Phi^{(2)}(\bx) = 0$, and
\begin{align}\vec\Phi^{(3)}_n(\bx) &=
\left( \begin{matrix} 0 \\ 0 \\ \sin \big( \frac{\pi n_1x_1}{L_1}\big)
  \sin \big( \frac{\pi n_2x_2}{L_2}\big)
\end{matrix}\right), \label{eq:eigvect4}
\end{align}
which is in the longitudinal direction. 

Equations (\ref{eq:Eo1})--(\ref{eq:Eo2}) take the explicit form 
\begin{align}
\label{eq:reference}
\vec{\bm{E}}^{o}(\vbx) = \sum_{ n\in\N^2_0}\sum_{s=1}^{m_n}
\Big[ & \delta_{s1} \vec\Phi_n^{(1)}(\bm x) (a^{+(1)}_{o,n}
  e^{i\beta_nx_3} + b^{+(1)}_{o,n}e^{-i\beta_nx_3}) \nonumber \\ &+
  \Big( \delta_{s2} \vec\Phi_n^{(2)}(\bm x) - \frac{i
      \lambda_n}{\beta_n} \delta_{s3} \vec\Phi_n^{(3)}(\bx)\Big)
  a_{o,n}^{+(2)}e^{i \beta_n x_3} \nonumber \\ &+ \Big(
    \delta_{s2} \vec\Phi_n^{(2)}(\bm x) + \frac{i \lambda_n}{\beta_n}
    \delta_{s3} \vec\Phi_n^{(3)}(\bx)\Big) b_{o,n}^{+(2)} e^{-i \beta_n x_3} \Big],
\quad \text{for } x_3 > -L,
\end{align}
and
\begin{align}
\vec{\bm{E}}^{o}(\vbx) = \sum_{n\in\N^2_0}\sum_{s=1}^{m_n}
\Big[& \delta_{s1} \vec\Phi_n^{(1)}(\bm x)
  b^{-(1)}_{o,n}e^{-i\beta_nx_3} + \nonumber \\ & \Big(\delta_{s2}
  \vec\Phi_n^{(2)}(\bm x) + \frac{i\lambda_n}{\beta_n} \delta_{s3}
  \vec\Phi_n^{(3)}(\bm x)\Big) b^{-(2)}_{o,n} e^{-i\beta_nx_3} \Big] ,
 \quad \text{for } x_3 < -L. \label{eq:refLeft}
\end{align}
The field $\vec{\bm E}^o$ is a superposition of transverse electric waves with
amplitudes $a_{n,o}^{+(1)}$ and $b_{n,o}^{\pm(1)}$, and transverse
magnetic waves with amplitudes $a_{n,o}^{+(2)}$ and
$b_{n,o}^{\pm(2)}$. The name transverse electric refers to the fact
that the third component of $\vec \Phi_n^{(1)}(\bx)$, corresponding to
the longitudinal electric field, equals zero. Similarly, the name
transverse magnetic refers to the fact that
\[
\vec{\bf e}_3 \cdot \curl \Big( \vec\Phi_n^{(2)}(\bm x) \pm \frac{i
  \lambda_n}{\beta_n} \vec\Phi_n^{(3)}(\bx) \Big) = \vec{\bf e}_3
\cdot \curl \vec\Phi_n^{(2)}(\bm x) = 0,
\]
and thus the longitudinal magnetic field is zero by Faraday's law.

The transverse electric mode amplitudes are given by
\begin{align}
\label{eq:ampA1}
a^{+(1)}_{o,n} =
-\frac{\omega\mu_o\lin\vec\Phi^{(1)}_n,\vbJ\rin}{2\beta_n \|\vec
  \Phi_n^{(1)}\|^2} e^{i\beta_nL}, \qquad b^{+(1)}_{o,n} =
\frac{\omega\mu_o\lin\vec\Phi^{(1)}_n,\vbJ\rin}{2\beta_n \|\vec
  \Phi_n^{(1)}\|^2}e^{i\beta_nL},
\end{align}
for $x_3 \in (-L,0)$ and by 
\begin{equation}
\label{eq:ampB1}
b^{-(1)}_{o,n} = \frac{\omega\mu_o \lin \vec\Phi^{(1)}_n,\vbJ
  \rin}{2\beta_n\|\vec \Phi_n^{(1)}\|^2} \left[ e^{i\beta_nL} -
  e^{-i\beta_nL} \right],
\end{equation}
for $x_3 < -L$.  Here $\lin \cdot, \cdot \rin$ denotes the inner
product in $\big(L^2(\Omega)\big)^3$ and $\| \cdot \|$ is the induced
norm.  The transverse magnetic mode amplitudes are
\begin{align}
 a^{+(2)}_{o,n} &= \left[-\frac{\om \mu_o\beta_n}{2k^2}\frac{\lin
     \vec\Phi^{(2)}_n,\vbJ\rin}{\|\vec \Phi_n^{(2)}\|^2} - \frac{i \om
     \mu_o}{2\lambda_n} \frac{\lin\vec\Phi^{(3)}_n,\vbJ\rin}{\|\vec
     \Phi_n^{(3)}\|^2}\right] e^{i\beta_nL}, \nonumber
 \\ b^{+(2)}_{o,n} &=
 \left[\frac{\omega\mu_o\beta_n}{2k^2}\frac{\lin\vec\Phi^{(2)}_n,\vbJ\rin}{
     \|\vec\Phi^{(2)}_n\|^2} +
   \frac{i\omega\mu_o}{2\lambda_n}\frac{\lin\vec\Phi^{(3)}_n,\vbJ\rin}{
     \|\vec\Phi^{(3)}_n\|^2} \right] e^{i\beta_nL}, \label{eq:ampAB2}
\end{align}
for $x_3 \in (-L,0)$ and 
\begin{equation}
 b^{-(2)}_{o,n} = \frac{\omega\mu_o\beta_n}{2k^2}\frac{\lin
   \vec\Phi^{(2)}_n,\vbJ \rin}{\|\vec\Phi^{(2)}_n\|^2} \left[
   e^{i\beta_nL} - e^{-i\beta_nL} \right] +
 \frac{i\omega\mu_o}{2\lambda_n}\frac{\lin\vec\Phi^{(3)}_n,\vbJ\rin}{
   \|\vec\Phi^{(3)}_n\|^2} \left[ e^{i\beta_nL} + e^{-i\beta_nL}
   \right], \label{eq:ampB2}
\end{equation}
for $x_3 < -L$. 
\subsection{The scattered field and radiation condition}
\label{sect:FP2}
The scattered field due to the reflectors supported in $D \subset W$
is defined by
\begin{equation}
\vec{\bm E}^{sc}(\vx) = \vec{\bm E}(\vx)-\vec{\bm E}^o(\vx),
\label{eq:SC1}
\end{equation}
where $\vec{\bm E}(\vx)$ is the solution of equation (\ref{eq:Maxwell1}),
with the electric permittivity tensor (\ref{eq:3}). Explicitly,
$\vec{\bm E}^{sc}$ satisfies
\begin{align}
\nonumber \curl\curl \vec{\bm{E}}^{sc}(\vx) -
k^2\vec{\bm{E}}^{sc}(\vx) &= k^2 V(\vx)
\vec{\bm{E}}(\vx) \qquad \vx \in W, \\
\label{eq:Scattered2}
\vbn(\vx) \times\vec{\bm{E}}^{sc}(\vx)  &= 0  \qquad \vx \in \pa W,
\end{align}
where 
\begin{equation}
\label{eq:defV}
V(\vx) = \frac{\eps(\vx)}{\eps_o} - I = 
1_{_D}(\vx) \big(\eps_r(\vx) - I\big)
\end{equation}
is the scattering potential. The radiation condition, which states that
the scattered field is bounded and outgoing away from the reflectors,
takes the form
\begin{align}
\hspace{-0.1in}\vec{\bm{E}}^{sc}(\vbx) = \sum_{n\in\N^2_0}\sum_{s=1}^{m_n}
\Big[& \delta_{s1} \vec\Phi_n^{(1)}(\bm x) b^{-(1)}_{n}e^{-i\beta_nx_3}
  \nonumber \\&+ \Big(\delta_{s2} \vec\Phi_n^{(2)}(\bm x) +
  \frac{i\lambda_n}{\beta_n} \delta_{s3} \vec\Phi_n^{(3)}(\bm
  x)\Big) b^{-(2)}_{n} e^{-i\beta_nx_3}\Big] ,\label{eq:radiationCond}
\end{align}
for locations $\vx = (\bx,x_3) \in W$ satisfying $x_3 < \inf\{x_3:
\vbx=(x_1,x_2,x_3) \in D\}$. Note the similarity of
(\ref{eq:radiationCond}) with (\ref{eq:refLeft}), the expression of
the reference field $\vec{\bm E}^{(o)}$ on the left of the source. The mode
amplitudes $b_n^{-(1)}$ and $b_n^{-(2)}$ contain the information about
the reflectors supported in $D$ and their expression follows from the
calculations in the next section.

\subsection{Solvability of the forward problem}
\label{sect:FP3}
Here we study the solvability of the forward scattering problem
(\ref{eq:Scattered2})--(\ref{eq:radiationCond}). We begin with the
derivation of the Green's tensor $\mathbb{G}(\vx,\vby)$ and then
restate the scattering problem as an electromagnetic
Lippmann-Schwinger equation, for which we can prove the Fredholm
property. The discussion assumes that the domain $D$ that supports the
reflectors does not touch the boundary, and that the scattering
potential $V$ is bounded, entrywise.

The Green's tensor $\mathbb{G}(\vx,\vby) \in \mathbb{C}^{3 \times 3}$
satisfies
\begin{align}
\nonumber \curl\curl \mathbb{G}(\vx,\vby) - k^2 \mathbb{G}(\vx,\vby)
&= -\delta(\vx -\vby) I \qquad \vx \in W, \\
\label{eq:Tensor2}
\vbn(\vx) \times\mathbb{G}(\vx,\vby) &= 0\qquad \vx \in \pa W,
\end{align}
where we recall that $I$ is the $3\times 3$ identity matrix, and the
curl is taken columnwise. In addition, each column of
$\mathbb{G}(\vx,\vby)$ satisfies a radiation condition similar to
(\ref{eq:radiationCond}) for $x_3 < y_3$, which says that the Green's
function is bounded and outgoing. The expression of $\mathbb{G}$ is
given in the next lemma, proved in appendix \ref{sect:proofLem2}.
\begin{lemma}
\label{lem.2}
Let $\vx \ne \vby$ and $\vx, \vby \in W$. The Green's tensor
$\mathbb{G}(\cdot, \vby)$ is given by
\begin{align}
\label{eq:TensorForm}
\mathbb{G}(\vx, \vby)= (\vec{G}_1,\vec{G}_2,\vec{G}_3)(\vx, \vby) +
\frac{1}{k^2} \vec\nabla \divv(\vec{G}_1,\vec{G}_2,\vec{G}_3)(\vx,
\vby),
\end{align}
with divergence taken columnwise. The vectors $\vec{G}_j$ with $j = 1,
\ldots, 3$ are defined by
\begin{equation}
  \label{eq:Gj}
\vec{G}_j(\vx, \vby) = \sum_{ n\in\N^2_0} \sum_{s=1}^{m_n}
\frac{\vec{\bm{e}}_j\cdot \vec\Phi^{(s)}_{n}(\bm{y})}{\|
  \vec\Phi^{(s)}_{n}\|^2}\big[e^{i\beta_n|x_3-y_3|} +
  (2\delta_{s3}-1)e^{-i\beta_n(x_3+y_3)}\big]
\frac{\vec\Phi^{(s)}_n(\bx)}{2 i \beta_n}.
\end{equation}
They satisfy equations
\begin{align}
\nonumber \Delta \vec{G}_j(\vx,\vby) + k^2 \vec{G}_j(\vx,\vby) &=
\delta(\vx -\vby) \vec{\bm{e}_j} \qquad \vx \in W, \\
\label{eq:Green2}
\vbn(\vx) \times \left[k^2\vec{G}_j(\vx,\vby) + \vec\nabla
  \divv\vec{G}_j(\vx,\vby)\right] &= 0\qquad \vx \in \pa W,
\end{align}
and a radiation condition similar to \eqref{eq:radiationCond} for $x_3
< y_3$, which says that the components of $\vec{G}_j$ are outgoing or
decaying waves.
\end{lemma}

To state the scattering problem (\ref{eq:Scattered2}) as a
Lippmann-Schwinger equation, we follow the approach in
\cite{Kirsc2007a}.  For a finite $\bar{L}\ge L$, we define the
truncated waveguide
\[
W_{\bar L} = (0,L_1)\times(0,L_2) \times(-\bar{L},0) \subset W,
\]
and introduce the space 
\[
H(\curll,W_{\bar{L}}):=\left\{\vbu \in (L^2(W_{\bar L}))^3: \curl \vbu
\in (L^2(W_{\bar L}))^3\right\},
\]
equipped with the inner product
\[
(\vec{\bm u},\vec{\bm v})_\icu = \int_{W_{\bar L}} d \vx \Big[ \vec{\bm
    u}(\vx) \cdot \overline{\vec{\bm v}(\vx)} + \curl \vec{\bm u}(\vx)
  \cdot \curl \overline{\vec{\bm u}(\vx)} \Big],
\]
where the bar denotes complex conjugate. The induced norm is $
\|\vec{\bm u}\|_\icu = \sqrt{(\vec{\bm u},\vec{\bm u})_\icu}.  $

 From \cite{Kirsc2007a} we known that $\mathcal{M}:\big(L^2(D)\big)^3
 \to H(\curll,W_{\bar{L}})$ defined by
\[
\mathcal{M}(\vec{\bm u})(\vx) = (k^2+\vec\nabla\divv)\int_D \, \frac{e^{i k
    |\vx-\vby|}}{4 \pi |\vx-\vby|} \vec{\bm u}(\vby)\d \vby,
\]
is a linear bounded mapping. Moreover, $\vec{\bm v} =
\mathcal{M}(\vec{\bm u})$ is the unique radiating variational solution
of $\curl\curl \vbv - k^2\vbv = k^2\vec{\bm u}$, meaning that
\begin{align}
\int_{W_{\bar{L}}} \left(\curl \vbv\cdot\curl \ol\vbphi -
  k^2\vbv\cdot\ol\vbphi\right)\d \vx = k^2\int_D\vbg\cdot\ol\vbphi \d \vx
\end{align}
for all $\vbphi \in H(\curll,W_{\bar{L}})$, with compact support in
$W_{\bar{L}}$. This result can be extended to our problem because the
difference of Green's functions $\vec{G}_j(\vx,\vby) - \frac{e^{i k
    |\vx-\vby|}}{4 \pi |\vx-\vby|}\vec{\bm e}_j$ is analytic and
satisfies
\[
\big(\Delta + k^2 \big) \Big(\vec{G}_j(\vx,\vby) - \frac{e^{i k
    |\vx-\vby|}}{4 \pi |\vx-\vby|}\vec{\bm e}_j\Big) = 0.
\]

Thus, the mapping $\mathcal{L}:\big(L^2(D)\big)^3 \to
H(\curll,W_{\bar{L}})$ defined by
\begin{equation}
\mathcal{L}(\vec{\bm u}) = (k^2+\vec\nabla\divv)\int_D
\big(\vec{G}_1,\vec{G}_2,\vec{G}_3\Big)(\cdot,\vby)\, \vec{\bm
  u}(\vby)\d \vby,
\label{eq:Ki1}
\end{equation}
is linear and bounded, and $\vec{\bm v} = \cL(\vbu)$ is the radiating
variational solution of the equation $\curl\curl \vbv - k^2\vbv =
k^2\vec{\bm u}$ in the waveguide.  We are interested in $\vec{\bm u} =
V \vec{\bm E}$, so that $\mathcal{L}(V \vec{\bm E})$ satisfies the partial
differential equation (\ref{eq:Scattered2}).  To show that this is
$\vec{\bm E}^{sc}$ it remains to check that $\mathcal{L}(V \vec{\bm E})$ satisfies
the perfectly conducting boundary conditions. This follows from the
boundary conditions in (\ref{eq:Green2}), because $D$ does not touch
the boundary, so we can write
\[
\vec{\bm n}(\vec{\bm x})\times \mathcal{L}(\vec{\bm u}) = \int_{D}
\vec{\bm n}(\vx) \times
(k^2+\vec\nabla\divv)\big(\vec{G}_1,\vec{G}_2,\vec{G}_3\big)(\cdot,\vby)\,
\vec{\bm u}(\vby)\d \vby = 0, \qquad \vx \in \partial W.
\]
We have now shown that $\vec{\bm E}^{sc}(\vx) = \mathcal{L}(V \vec{\bm E})$, or
equivalently, that it solves the Lippmann-Schwinger equation
\begin{equation}
\vec{\bm E}^{sc}(\vx) = (k^2+\vec\nabla\divv)\int_D
\big(\vec{G}_1,\vec{G}_2,\vec{G}_3\Big)(\cdot,\vby)\, V(\vby) \vec{\bm E}(\vby)\d \vby.
\label{eq:Ki2}
\end{equation}

The next Theorem proves a G\r{a}rding inequality from which we can
conclude the Fredholm property.
\begin{theorem}
\label{thm.1}
There exists a compact operator $\mathcal{K}:H(\curll,W_{\bar{L}}) \to
H(\curll,W_{\bar{L}})$ and a positive constant $C$ such that
\begin{equation}
\Re \big(\vec{\bm u} - \mathcal{L}(V\vec{\bm u}) + \mathcal{K} \vbu
,\vec{\bm u}\big)_\icu \ge C \|\vec{\bf u}\|^2_\icu, \qquad \forall
\, \vbu \in H(\curll,W_{\bar{L}}).
\label{eq:Gard}
\end{equation}
Therefore, $I - \mathcal{L}(V\cdot)$ is a Fredholm operator.
\end{theorem}

\begin{proof}
Let us define an auxilliary operator $\mathcal{L}_o:\big(L^2(D)\big)^3 \to
H(\curll,W_{\bar{L}})$, 
\begin{equation}
\mathcal{L}_o(\vec{\bm u}) = (-1+\vec\nabla\divv)\int_D
\big(\vec{\mathcal{G}}_1,\vec{\mathcal{G}}_2,\vec{\mathcal{G}}_3\Big)(\cdot,\vby)\,
\vec{\bm u}(\vby)\d \vby,
\label{eq:PF1}
\end{equation}
where $\vec{\mathcal{G}}_j$ solve
\begin{equation}
\Delta \vec{\mathcal{G}}_j(\vx,\vby) - \vec{\mathcal{G}}_j(\vx,\vby) =
  \delta(\vx-\vby) \vec{\bf e}_j, \quad \vx \in W_{\bar{L}}.
\end{equation}
These are like the partial differential equations in
(\ref{eq:Green2}), with $k$ replaced by the imaginary number $i$.
From the analysis in \cite{Kirsc2007a}, which applies to imaginary
wavenumbers like $i$, we obtain that $\cL_o$ is a bounded linear
operator and $\vbu = \cL_o(\vec{\bm f})$ is the weak solution of
$\curl \curl \vbu + \vbu = - \vec{\bm f}$. Explicitly, we have for all
$\vbphi \in H(\curll,W_{\bar{L}})$,
\begin{align}
\Big( \cL_o(\vec{\bm f}),\vbphi \Big)_\icu &= \int_{W_{\bar{L}}} d \vx
\Big[\curl \cL_o(\vec{\bm f}) \cdot\curl \ol\vbphi + \cL_o(\vec{\bm f
    }) \cdot\ol\vbphi\Big] \nonumber \\ &= -\int_D d \vx \, \vec{\bm f
  }\cdot\ol\vbphi- \int_{\pa W_{\bar{L}}} ds \, \Big[\vec{\bm n}
  \times \curl \cL_o(\vec{\bm f})\Big] \cdot \Big[(\vec{\bm n} \times
  \ol\vbphi) \times \vec{\bm n}\Big],
\label{eq:IBP}
\end{align}
where we used the integration by parts result in \cite[Theorem
  3.31]{Monk2003a}.

Using this auxiliary operator we write
\begin{align*}
  \Big( \vbu - \cL(V\vbu),\vbu \Big)_\icu &= \Big( \vbu -
  \cL_o(V\vbu),\vbu \Big)_\icu \hspace{-0.05in} - \Big( (\cL-\cL_o)(V
  \vbu),\vbu)_\icu \nonumber \\ &= \|\vbu\|_\icu^2 - \Big(
  \cL_o(V\vbu),\vbu \Big)_\icu \hspace{-0.05in}- \Big( (\cL-\cL_o)(V
  \vbu),\vbu \Big)_\icu,
\end{align*}
and from (\ref{eq:IBP}) with $\vec{\bm f} = V \vbu$ and $\vbphi =
\vbu$, we get
\begin{align*}
  \Big( \vbu - \cL(V\vbu),\vbu \Big)_\icu =&\|\vbu\|_\icu^2 + \int_D d
  \vx \, (\eps_r(\vx)-I) \vbu \cdot \ol\vbu \\ &\hspace{-0.8in}+ \int_{\pa
    W_{\bar{L}}} ds \, \Big[\vec{\bm n} \times \curl \cL_o(
      V\vbu)\Big] \cdot \Big[(\vec{\bm n} \times \ol\vbu) \times
    \vec{\bm n}\Big] - \Big( (\cL-\cL_o)(V \vbu),\vbu \Big)_\icu.
\end{align*}
Here we used the expression (\ref{eq:defV}) of $V$. Because $\eps_r$
is positive definite by assumption, we conclude that there exists a
positive constant $C$ such that
\begin{align*}
  \|\vbu\|_\icu^2 + \int_D d
  \vx \, (\eps_r(\vx)-I) \vbu \cdot \ol\vbu \ge C \|\vbu\|_\icu^2, \quad
  \forall \, \vbu \in H(\curll,W_{\bar L}).
\end{align*}
Substituting in the equation above and introducing the linear
operators $\mathcal{K}_1$ and $\mathcal{K}_2$ from $H(\curll,W_{\bar
  L})$ to $H(\curll,W_{\bar L})$, defined by
\begin{align}
  \mathcal{K}_1(\vbu) &= (\cL-\cL_o)(\vbu), \label{eq:K1}
  \\ \mathcal{K}_2(\vbu) &= -\int_{\pa W_{\bar{L}}} ds \, \Big[\vec{\bm
      n} \times \curl \cL_o( V\vbu)\Big] \cdot \Big[(\vec{\bm n}
    \times \ol\vbu) \times \vec{\bm n}\Big], \label{eq:K2}
\end{align}
we obtain
\begin{equation}
  \Re \Big(\vbu - \cL(V \vbu) + \mathcal{K}_1(\vbu) +
  \mathcal{K}_2(\vbu),\vbu\Big)_\icu \ge C \|\vbu\|_\icu^2, \quad
  \forall \, \vbu \in H(\curll,W_{\bar L}). \label{eq:K3}
\end{equation}
Result (\ref{eq:Ki2}) follows once we show that $\mathcal{K}_1$ and
$\mathcal{K}_2$ are compact operators.

Since the differences $\vec{G}_j(\vx,\vby) - \frac{e^{i k
    |\vx-\vby|}}{4 \pi |\vx-\vby|}\vec{\bm e}_j$ and
$\vec{\mathcal{G}}_j(\vx,\vby) - \frac{e^{- |\vx-\vby|}}{4 \pi
  |\vx-\vby|}\vec{\bm e}_j$ are analytic, we conclude that the
singularity of the kernel in $\cL-\cL_o$ is as strong as that of
$\Big[\frac{e^{i k |\vx-\vby|}}{4 \pi |\vx-\vby|} - \frac{e^{-
      |\vx-\vby|}}{4 \pi |\vx-\vby|}\Big]I$. Thus, we can use the
results in \cite{Kirsc2007a} to conclude that $\mathcal{K}_1$ is a
compact operator.

To prove that $\mathcal{K}_2$ is compact, let us consider a
neighborhood $\Gamma$ of the boundary $\pa W_{\bar L}$, such that
$\Gamma \subset W_{\bar L}$ and $\Gamma$ does not intersect the
support $D$ of the scattering potential. We define the operator
$\mathcal{T}$ from $(L^2(D))^3$ to $(H^s(\Gamma))^3$,
with $s > 1$, by restricting $\curl \cL_o(\vec{\bm f})$ to $\Gamma$,
for all $\vec{\bm f} \in (L^2(D))^3$,
\begin{equation}
  \mathcal{T}(\vec{\bm f}) = - \int_{D} d \vby\, \nabla \times
  \Big(\vec{\mathcal{G}}_1,\vec{\mathcal{G}}_2,\vec{\mathcal{G}}_3\Big)(\cdot, \vby)
  \vec{\bm f}(\vby), \quad \mbox{in}~ \Gamma.
\end{equation}
This operator is compact because its kernel is an analytic function on
$\Gamma \times D$.  Define also the trace space
\begin{equation*}
  H^{-1/2}_\idiv(\pa W_{\bar L}) = \Big\{ \vec{\bm f} \in
  \Big(H^{-1/2}(\pa W_{\bar L})\Big)^3: ~\exists~ \vbu \in
  H(\curll, W_{\bar L}) ~ \mbox{satisfying}~ \vec{\bm n} \times
  \vbu|_{\pa W_{\bar L}} = \vec{\bm f}\Big\},
\end{equation*}
with norm
\[
\|\vec{\bm f}\|_{H^{-1/2}_\idiv(\pa W_{\bar L})} = \inf_{\vbu \in
  H(\curll, W_{\bar L}), \vec{\bm n} \times \vbu|_{\pa W_{\bar L}} =
  \vec{\bm f}} \|\vbu\|_\icu.
\]
It is shown in \cite[Section 3.5]{Monk2003a} that $H^{-1/2}_\idiv(\pa
W_{\bar L})$ is a Banach space. Due to the compactness of
$\mathcal{T}$, the mapping $\vbu \to \vec{\bm n} \times
\mathcal{T}(V \vbu)|_{\partial W_{\bar L}}$ is a compact operator from
$H(\curll,W_{\bar L})$ to $H^{-1/2}_\idiv(\pa W_{\bar L})$. Note that
the mapping $\vbu \to V \vbu$ is bounded from $H(\curll,W_{\bar L})$
to $(L^2(D))^3$ and $\mathcal{T}(\vbu) \to \vec{\bm n} \times
\mathcal{T}(\vbu)|_{\partial W_{\bar L}}$ is bounded from
$(H^s(D))^3$ to $H^{-1/2}_\idiv(\pa W_{\bar L})$.  We also
have from \cite[Section 3.5]{Monk2003a} that $\vbu \to (\vec{\bm
  n} \times \vbu|_{\pa W_{\bar L}}) \times \vec{\bm n}$ is a
linear bounded mapping from $H(\curll,W_{\bar L})$ to
$H^{-1/2}_\icu(\pa W_{\bar L})$, the dual space of $H^{-1/2}_\idiv(\pa
W_{\bar L})$.

To show that $\mathcal{K}_2$ is compact, let $\{\vbu_j\}$ be a
sequence in $H(\curll,W_{\bar L})$ that converges weakly to $0$, and
prove that $\{\mathcal{K}_2(\vbu_j)\}$ converges strongly to $0$ in
$H(\curll,W_{\bar L})$. Indeed we have
\begin{align*}
  \|\mathcal{K}_2 (\vbu_j)\|_\icu &= \sup_{\vec{\bm v} \in
    H(\curll,W_{\bar L})\setminus\{0\}} \frac{\Big|
    (\mathcal{K}_2(\vbu_j),\vec{\bm v})_\icu \Big|}{\|\vec{\bm
      v}\|_\icu} \nonumber \\ 
      &\leq \sup_{ \vbv\in H(\curll,W_{ \bar
      L})\setminus \{0\}} \frac{ \| \vbn \times \curl
    \cL_o(V\vbu_j)\|_{H^{-1/2}_\idiv(\pa W_{\bar L})} \| \vbn\times
    \vbv \times \vbn \|_{H^{-1/2}_\icu(\pa W_{\bar L})}}{\|
    \vbv\|_\icu} \nonumber \\ 
    &\leq C\| \vbn \times \curl
  \cL_o(V\vbu_j)\|_{H^{-1/2}_{\div}(\pa W_{\bar L})} \nonumber \\ &=
  C\| \vbn \times \mathcal{T}(V\vbu_j) \|_{H^{-1/2}_\idiv(\pa W_{\bar L})} \\
  & \to 0, \qquad \mbox{as} ~ j \to \infty.
\end{align*}
where the first line is a definition, the second line follows by
duality, the third line is due to the boundedness of the mapping $\vbv
\to \vbn \times \vbv \times \vbn$ and the fourth line is by the
definition of $\mathcal{T}$. The convergence to zero is by the
compactness of the mapping $\vbu \to \vec{\bm n} \times \mathcal{T}(V
\vbu)|_{\partial W_{\bar L}}$.

We have now proved the G\r{a}rding inequality (\ref{eq:Gard}), with
$\mathcal{K} = \mathcal{K}_1 + \mathcal{K}_2$. We obtain from it
that $I - \cL(V \cdot)$ is the sum of the coercive operator $I - \cL(V
\cdot) - \mathcal{K}$ and the compact operator $\mathcal{K}$. Thus,
$I-\cL(V\cdot)$ is a Fredholm operator \cite{McLea2000}.
\end{proof}

We conclude the discussion on the solvability of the forward problem
with the remark that when $\eps_r$ is $C^1$, one can extend the
results in \cite{Kirsc2007a} to prove uniqueness of solution of
equation (\ref{eq:Ki2}). The existence of the solution follows from
the Fredholm property.

\section{Data model}
\label{sect:FP4}
Since the array is far away from the support $D$ of the scattering
potential $V$, at coordinate $x_3 = -L$,  the results in
the previous section give
\begin{equation}
\vec{\bm E}^{sc}(\vx) \approx k^2 \int \mathbb{G}^P(\vx,\vby)\, \vec{\bm u}(\vby) \d \vby,
\quad \vx = (\bx,-L).
\label{eq:B01}
\end{equation}
Here
\begin{equation}
\vec{\bm u}(\vby) = V(\vby) \vec{\bm E}(\vby), 
\label{eq:B02}
\end{equation}
is an effective source supported in $D$, representing the wave emitted
by the unknown reflectors illuminated by the field $\vec{\bm E}(\vby)$.  The
approximation in (\ref{eq:B01}) is because we replaced the Green
tensor $\mathbb{G}$ defined in Lemma \ref{lem.2} by its approximation
$\mathbb{G}^P$ which neglects the evanescent waves. Explicitly, if we
denote by $P$ the set of indexes of the propagating modes
\[
P = \{n \in \N_o^2 : \lambda_n < k^2\},
\]
we have
\begin{equation}
\mathbb{G}^P(\vx, \vby)= (\vec{G}^P_1,\vec{G}^P_2,\vec{G}^P_3)(\vx,
\vby) + \frac{1}{k^2} \vec\nabla
\divv(\vec{G}^P_1,\vec{G}^P_2,\vec{G}^P_3)(\vx, \vby),
\label{eq:B2}
\end{equation}
with
\begin{equation}
  \label{eq:B3}
\vec{G}^P_j(\vx, \vby) = \sum_{ n\in P}
\sum_{s=1}^{m_n}\frac{\vec{\bm{e}}_j\cdot
  \vec\Phi^{(s)}_{n}(\bm{y})}{\| \vec\Phi^{(s)}_{n}\|^2}
\big[e^{i\beta_n(y_3+L)} +
  (2\delta_{s3}-1)e^{i\beta_n(L-y_3)}\big]
\frac{\vec\Phi^{(s)}_n(\bx)}{2 i \beta_n},
\end{equation}
where we used that $x_3 = -L$ at the array.

Let us denote by $\mathcal{S}_q$ the linear mapping from the effective
source (\ref{eq:B02}) to the $q-$th component of the scattered field
at the array
\begin{equation}
\big[\mathcal{S}_q(\vec{\bm u})\big](\bx) = k^2\int \vec{\bm e}_q
\cdot \mathbb{G}^P((\bx,-L),\vby)\, \vec{\bm u}(\vby) \d \vby, 1 \le q \le 3. 
\label{eq:B03}
\end{equation}
Since the support of the source (\ref{eq:B02}) is included in $D$, we
may seek to reconstruct the domain $D$ by inverting approximately $\mathcal{S}_q$.
The mapping that takes the scattering potential $V$ to the
measurements is nonlinear, because the scattered field $E^{sc}$ enters
the definition (\ref{eq:B02}). Thus, we linearize it, meaning that we
make the single scattering (Born) approximation
\begin{equation}
\vec{\bm u}(\vby) \approx V(\vby) \vec{\bm E}^o(\vby).
\label{eq:B04}
\end{equation}
We denote by $\mathcal{B}$ the linear mapping from the scattering
potential $V$ to the effective source
\begin{equation}
\big[\mathcal{B}(V)](\vby) = V(\vby) \vec{\bm E}^o(\vby).
\label{eq:B05}
\end{equation}
Then, the forward map $\cF_q$ from the scattering potential $V$ to the
$q-th$ component of the electric field measured at the array is the
composition of the mappings in (\ref{eq:B04}) and (\ref{eq:B05}),
\begin{equation}
\cF_q(V) = \mathcal{S}_q \circ \mathcal{B} (V)
\label{eq:B4}
\end{equation}

The data are denoted by $d_q(\bx)$, for components $q = 1, \ldots, Q$,
with $Q \le 3$, and $\bx \in A$, the aperture of the array, which is
a subset of the waveguide cross-section $\Omega$.

\section{Imaging}
\label{sect:imag}
Let ${\bf d}$ be the data vector, with entries given by $d_q(\bx)$ for
all $\bx$ in $A$ and $q = 1, \ldots, Q$. Let also $\bV$ be the
reflectivity vector consisting of the unknown components of the
scattering potential $V$, discretized in the imaging window $D_I$ that
contains the unknown support $D$. Then, we can state the imaging
problem as finding an approximate solution $\bV$ of the linear system
of equations
\begin{equation}
{\bf d} = \bF \bV.
\label{eq:B5}
\end{equation}
The reflectivity to data matrix $\bF$ is defined by the discretization
of the forward mapping (\ref{eq:B4}).

The system of equations (\ref{eq:B5}) is usually undertermined, so to
find a unique approximation we regularize the inversion by minimizing
either the $\ell_2$ or the $\ell_1$ norm of $\bV$. The first
regularization is related to the reverse time migration approach, as
described in section \ref{sect:TR}.  The imaging with $\ell_1$
minimization is discussed in section \ref{sect:L1}.

\subsection{Reverse time migration}
\label{sect:TR}
The minimum $\ell_2$ norm solution of (\ref{eq:B5}) is 
\begin{equation}
\bV = \bF^\dagger {\bf d},
\label{eq:TR1}
\end{equation}
where $\bF^\dagger$ is the pseudo-inverse of $\bF$. If $\bF$ is full
row rank, $\bF^\dagger = \bF^\star (\bF \bF^\star)^{-1}$, where the
superscript denotes the adjoint. Moreover, if the rows of $\bF$ are
nearly orthogonal, which requires proper placement of the receiver
locations in the array aperture $A$, at distance of the order of the wavelength, matrix
$\bF \bF^\star$ is nearly diagonal, so by replacing $\bF^\dagger$ in
(\ref{eq:TR1}) with $\bF^\star$ we get a similar answer, up to
multiplicative factors. This replacement does not affect
the support of the reconstruction and we denote the result by 
\begin{equation}
\bV^{^{\tiny \mbox{TR}}} = \bF^\star {\bf d},
\label{eq:TR2}
\end{equation}
with superscript TR for ``time reversal".

To explain where time reversal comes in, let us compute the adjoint of the
forward mapping \eqref{eq:B4}. Before discretizing the imaging window
we have
\begin{align*}
\big(\mathcal{F}(V),{\bf d}\big) &= \sum_{q=1}^Q \sum_{\bx \in A}
\big[\mathcal{F}_q(V)\big](\bx) \overline{d_q(\bx)} \nonumber \\ &=
k^2 \sum_{q=1}^Q \sum_{\bx \in A} \int d \vby \, \vec{\bf e}_q \cdot
\mathbb{G}^P((\bx,-L),\vby) V(\vby) \vec{\bm E}^o(\vby) \overline{d_q(\bx)},
\end{align*}
by the definition \eqref{eq:B4} of the forward map and equation \eqref{eq:B03}.
We rewrite this as  
\begin{align}
\big(\mathcal{F}(V),{\bf d}\big) &= \sum_{l=1}^3 \int d \vby \,
\big[V(\vby) \vec{\bm E}^o(\vby)\big]_l \, \Big[k^2 \sum_{q=1}^Q \sum_{\bx
    \in A} \mathbb{G}^P_{lq}(\vby, (\bx,-L))
  \overline{d_q(\bx)}\Big], \label{eq:TR3}
\end{align}
using the Rayleigh-Carson reciprocity relation $\mathbb{G}^P(\vbx,\vby) =\big[ \mathbb{G}^P(\vby,\vx)\big]^T$
of the Green's tensor. 
The last factor, in the square brackets, is the electric field
evaluated at points $\vby$ in the imaging window $D_I$, due to a
source at the array which emits the data recordings $d_q$ reversed in
time. The time reversal is equivalent to complex conjugation in the
Fourier domain. The adjoint of the forward map follows from
(\ref{eq:TR3}),
\begin{align}
\big(\mathcal{F}(V),{\bf d}\big) &= \sum_{l,m=1}^3 \int V_{lm}(\vby)
E^o_m(\vby) \Big[k^2 \sum_{q=1}^Q \sum_{\bx \in A}
  \mathbb{G}^P_{lq}(\vby, (\bx,-L)) \overline{d_q(\bx)}\Big]
\nonumber= \big(V,\cF^\star({\bf d})),
\end{align}
where the inner product in the right hand side is 
\[
\big(V,U\big) = \int d\vby \, \mbox{trace} \big[V(\vby) U(\vby)\big],
\]
for any complex valued matrix $U$. Recall that $V(\vby)$ is Hermitian.
Thus, $\cF^\star(\bd)$ is a $3 \times 3$ complex matrix valued field,
with components 
\begin{equation}
\big[\cF^\star(\bd)\big]_{ml}(\vby) = \Big[k^2 \sum_{q=1}^Q
  \sum_{\bx \in A} \mathbb{G}^P_{lq}(\vby, (\bx,-L))
  \overline{d_q(\bx)}\Big]E_m^o(\vby).
\label{eq:TR4}
\end{equation}
The right hand side in the imaging formula (\ref{eq:TR2}) is the discretization of  (\ref{eq:TR4}) over points $\vby$ in the imaging
window. 

In the particular case of a diagonal scattering potential $V(\vby)$,
which corresponds to the coordinate axes being the same as the
principal axes of the dielectric material in the support of the
reflectors, 
the adjoint operator acts from the data space to the space of diagonal,
positive definite matrices. The reconstruction is given by 
\begin{equation}
V^{^{\tiny \mbox{TR}}}_{ll}(\vby) = \Big[k^2 \sum_{q=1}^Q
  \sum_{\bx \in A} \mathbb{G}^P_{lq}(\vby, (\bx,-L))
  \overline{d_q(\bx)}\Big]E_l^o(\vby),
\label{eq:TR5}
\end{equation}
where $\vby$ are the discretization points in $D_I$.  Moreover, if the
material is isotropic, so that $V$ is a multiple of the identity, the
reconstruction is $V^{^{\tiny \mbox{TR}}} I$, with
\begin{equation}
V^{^{\tiny \mbox{TR}}}(\vby) = \sum_{l=1}^3 \Big[k^2
  \sum_{q=1}^Q \sum_{\bx \in A} \mathbb{G}^P_{lq}(\vby, (\bx,-L))
  \overline{d_q(\bx)}\Big]E_l^o(\vby).
\label{eq:TR6}
\end{equation}
None of these formulae are quantitative approximations of $V$, so we
may drop the factor $k^2$ and display their absolute values at points $\vby$ in the imaging window $D_I$. 
The estimate of the 
support $D$ of $V$ is given by the subset in $D_I$ where the displayed values are  large.

\subsection{Imaging with $\ell_1$ optimization}
\label{sect:L1}
To incorporate the prior information that the reflectors have small
support in the imaging window, we may reconstruct the scattering
potential using $\ell_1$ optimization. This means solving the
optimization problem 
\begin{equation}
\min \|\bV\|_{\ell_1} \quad \mbox{such that} \quad \bd = \bF \bV.
\label{eq:L1}
\end{equation}
The equality constraint may be replaced by the inequality $\|\bd - \bF
\bV\|_{\ell_2}^2 \le $ some user defined tolerance, which deals better
with measurement and modeling noise. The $\ell_1$ optimization is carried with the  
cvx package  ``http://cvxr.com/cvx/".

\section{Numerical simulations}\label{sect:num}

We present in this section examples of reconstructions of   the reflectors  with 
reverse time migration and $\ell_1$ optimization. The simulations are for a waveguide
with cross-section  $\Omega =\big(0,13.9\lambda)\times(0,14.2\lambda\big)$, and the array is at 
distance  $L = 41.8 \lambda$ from the end wall, where $\lambda$ is the wavelength.
The source density  in~\eqref{eq:Maxwell1} is 
\begin{equation}
\label{eq:JSource}
\vbJ(\bx) =  \vec{p}\, \delta\Big(\bx - (6.95,7.1) \lambda\Big),
\end{equation}
for constant vector $\vec{p}$, 
and the receiver sensors are located at uniform spacing of  approximately $\lambda/18$ in the array aperture $A$.
We present results with full aperture, where $A = \Omega$ and with $75\%$ aperture, where 
$A \subset \Omega$ is a rectangle of sides $10.5 \lambda$ and $10.65 \lambda$,
with center at the waveguide axis. The receivers measure only the  $2-$nd component  of $\vec{\bm E}^{sc}$ .
We compared the results with those obtained from all components of $\vec{\bm E}^{sc}$ at the array, 
and the images were essentially the same. 

The images displayed in Figures \ref{fi:2}--\ref{fi:5} are obtained with an approximation of the formulae 
in section \ref{sect:imag}, where only a subset of the $648$ propagating modes  are used. This is because 
in practice the sensors record over a finite time window, and only the modes that propagate fast enough 
to arrive at the array during the duration of the measurements contribute. The polarization vector $\vec{p}$
in \eqref{eq:JSource} equals $(0,1,0)^T$ in the simulations with isotropic permittivity and $(1,1,1)^T$ 
in the case of anisotropic permittivity. 

 \begin{figure}[h!!t!!!b!!!!]
\centering
{\includegraphics[width=6cm]{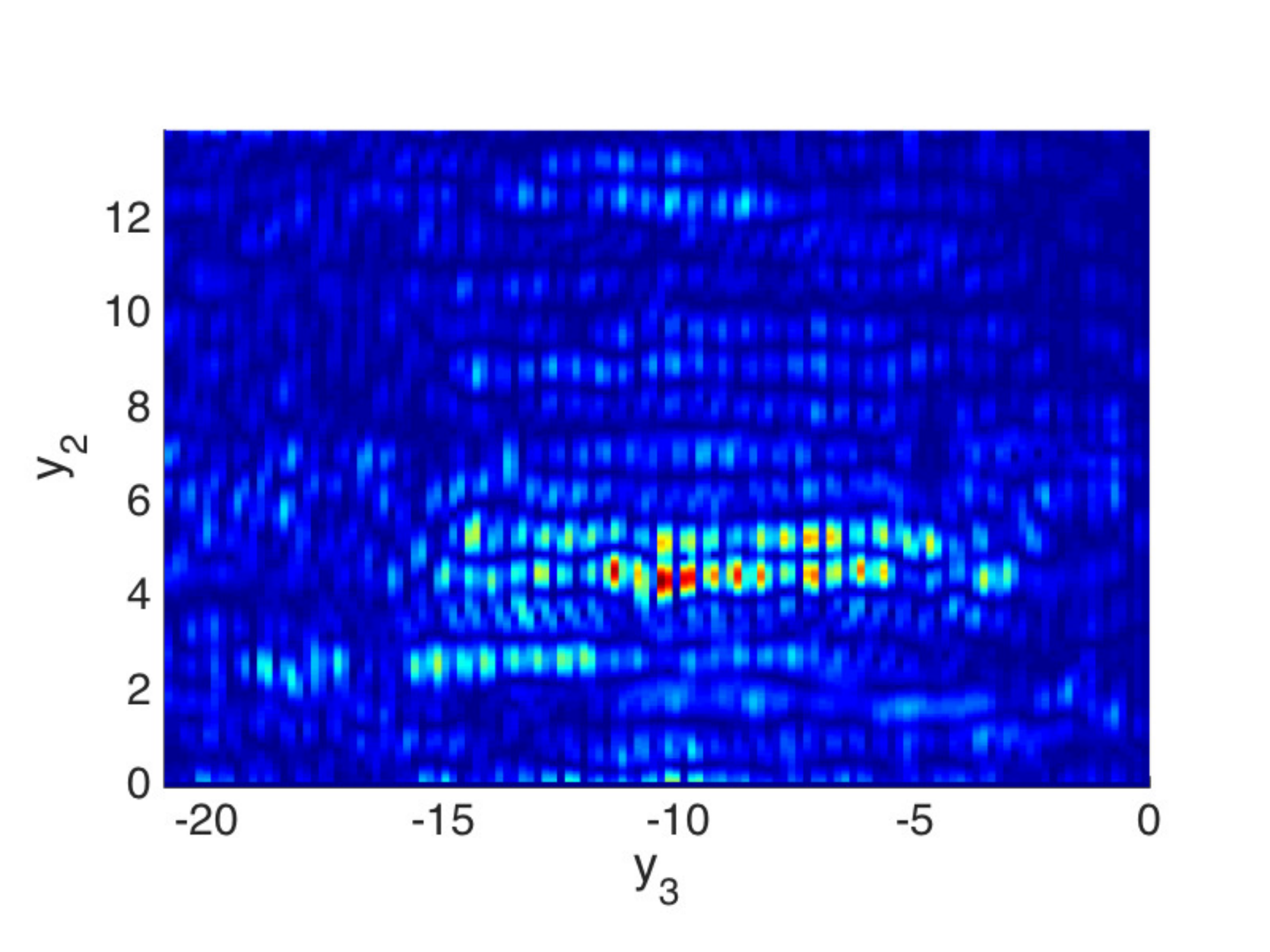}}
\hspace{0cm}
{\includegraphics[width=6cm]{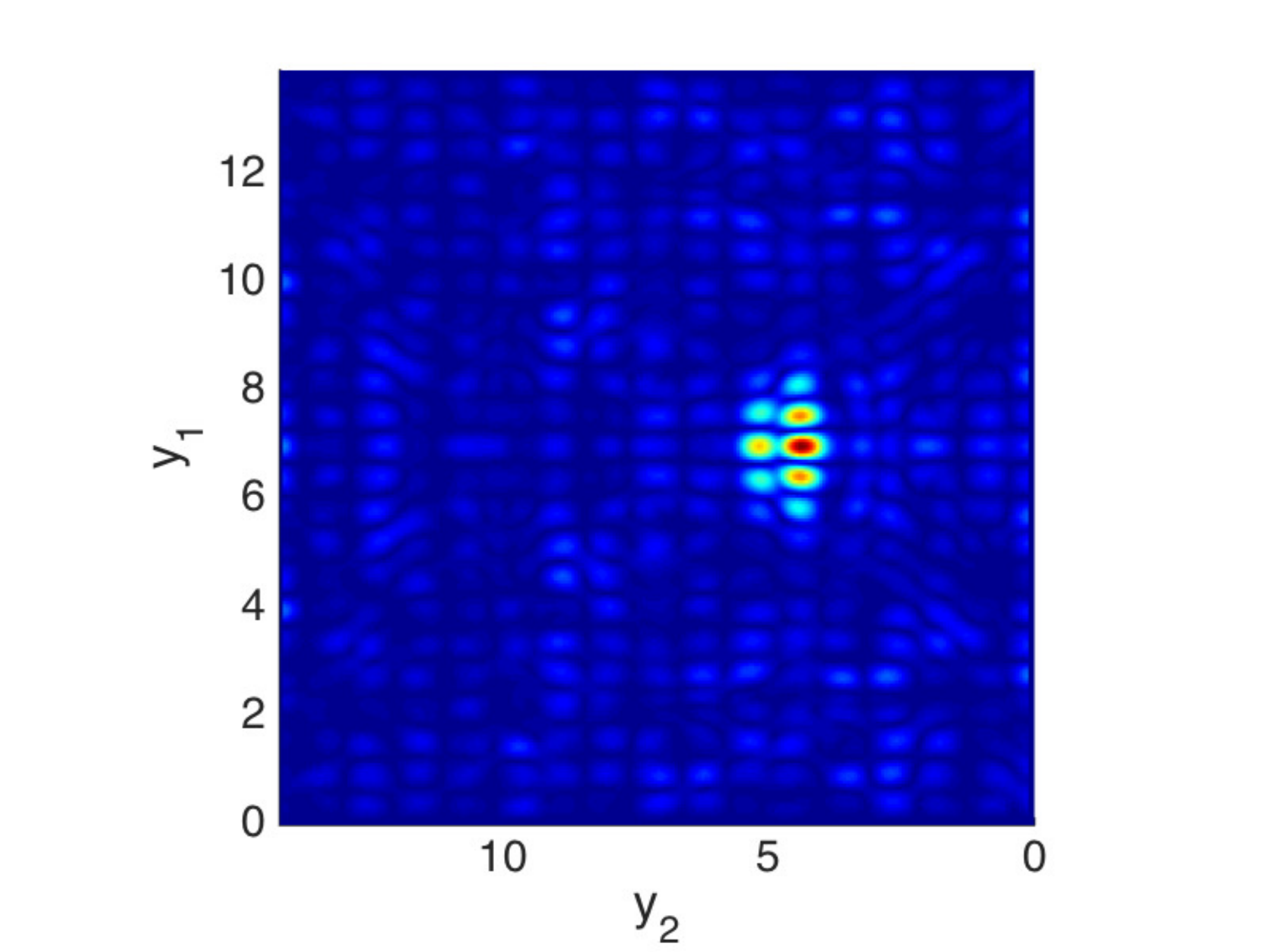}}\\
{\includegraphics[width=6cm]{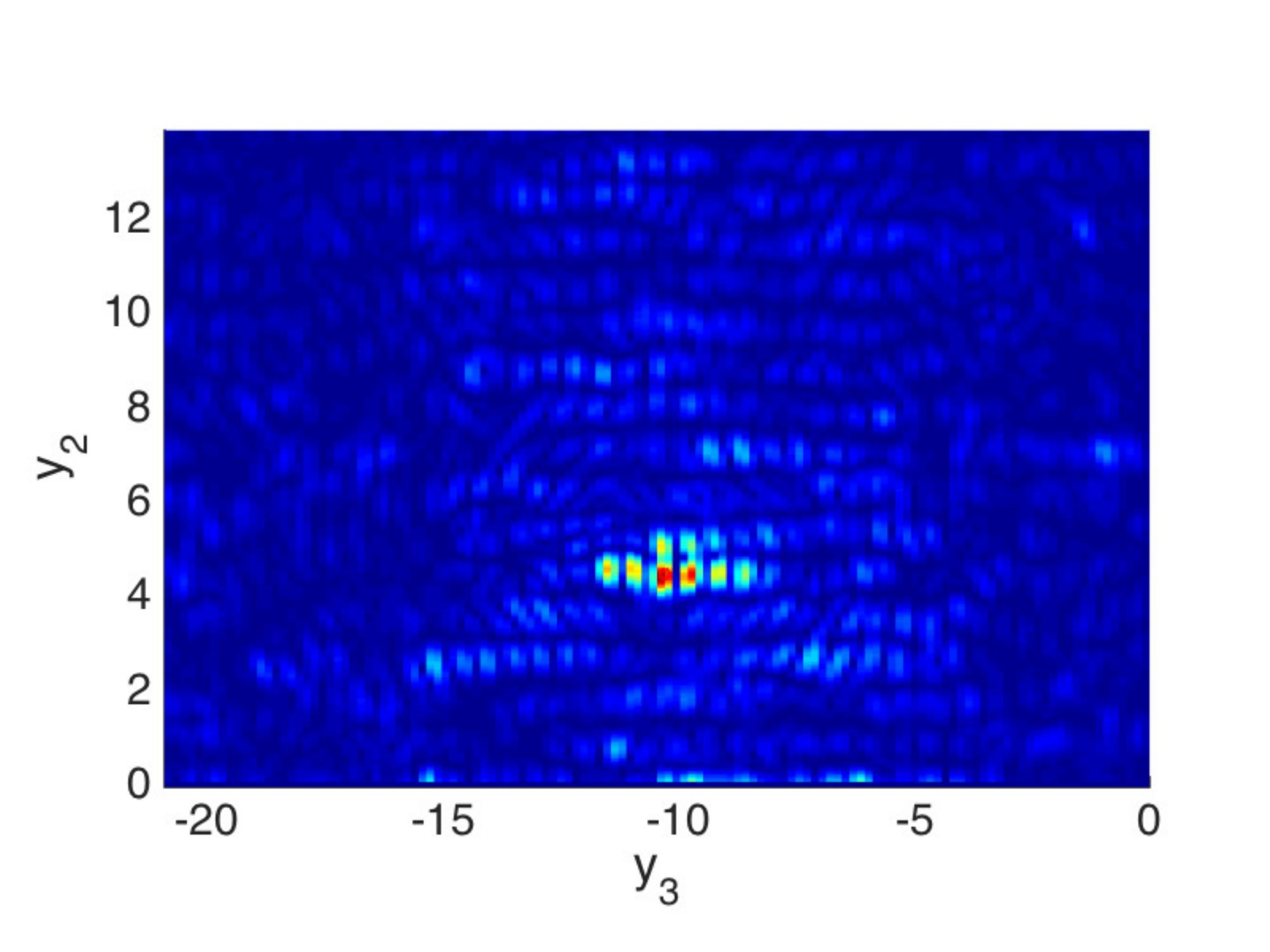}}
\hspace{0cm}
{\includegraphics[width=6cm]{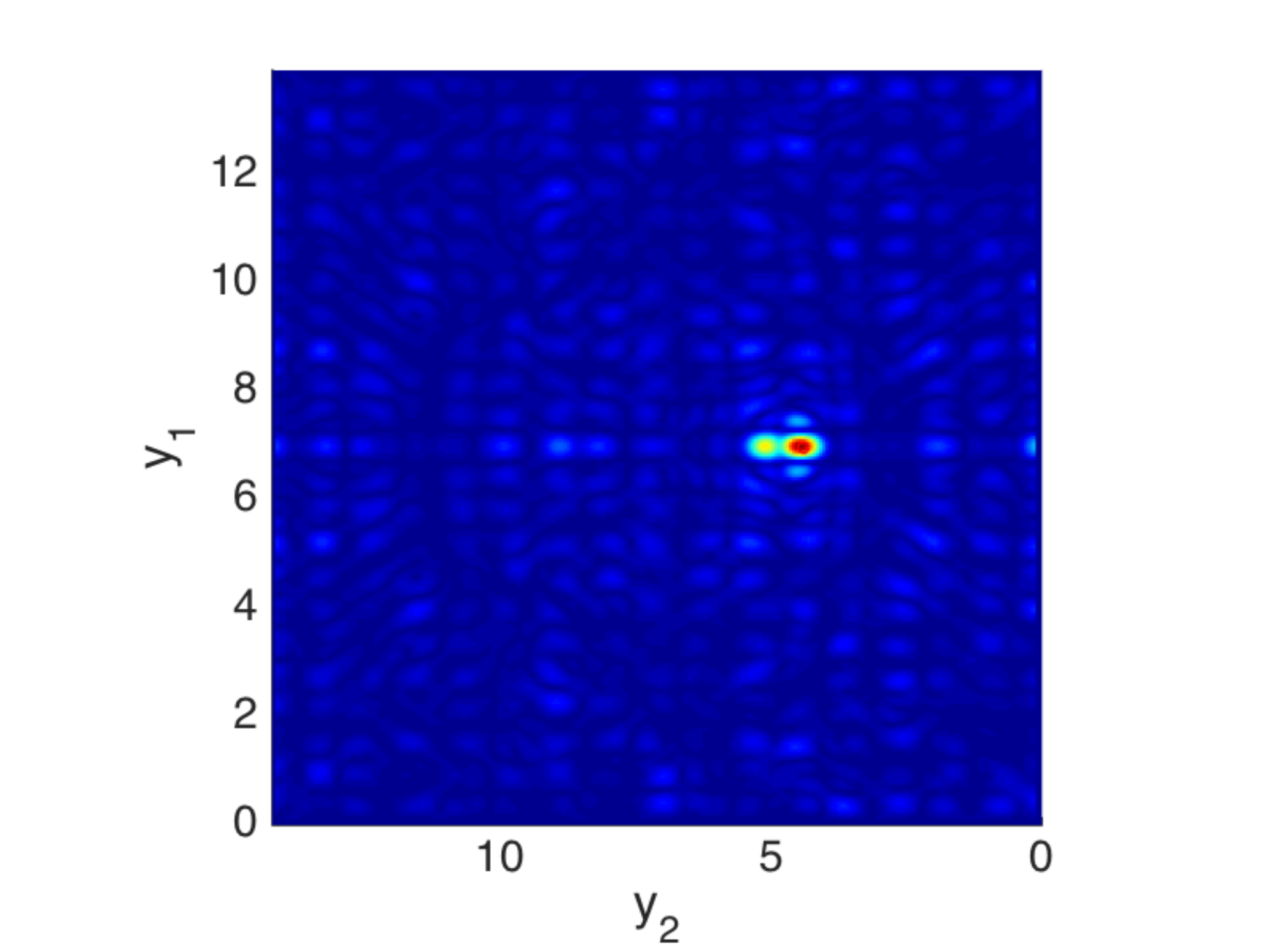}}\\
{\includegraphics[width=6cm]{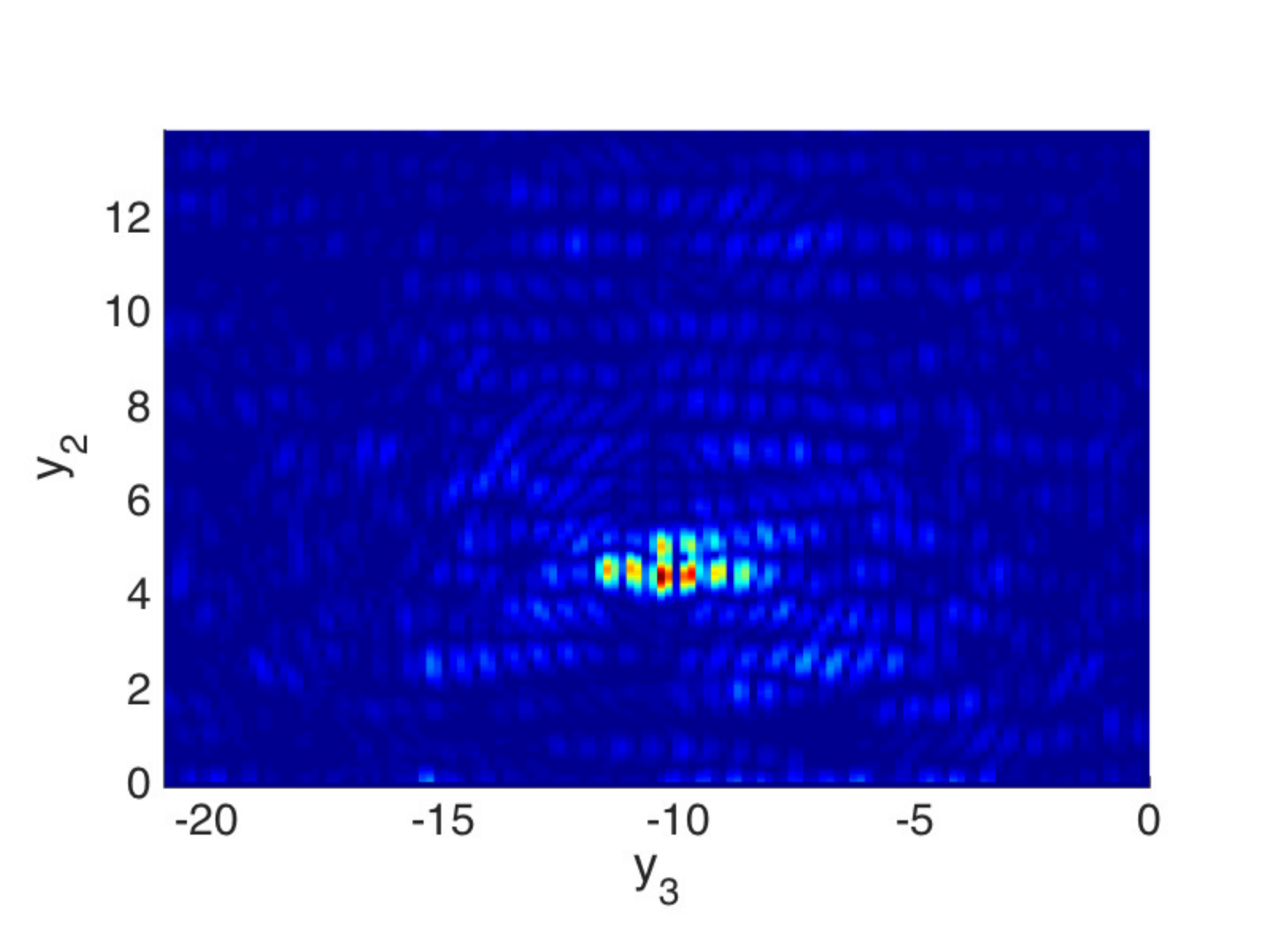}}
\hspace{0cm}
{\includegraphics[width=6cm]{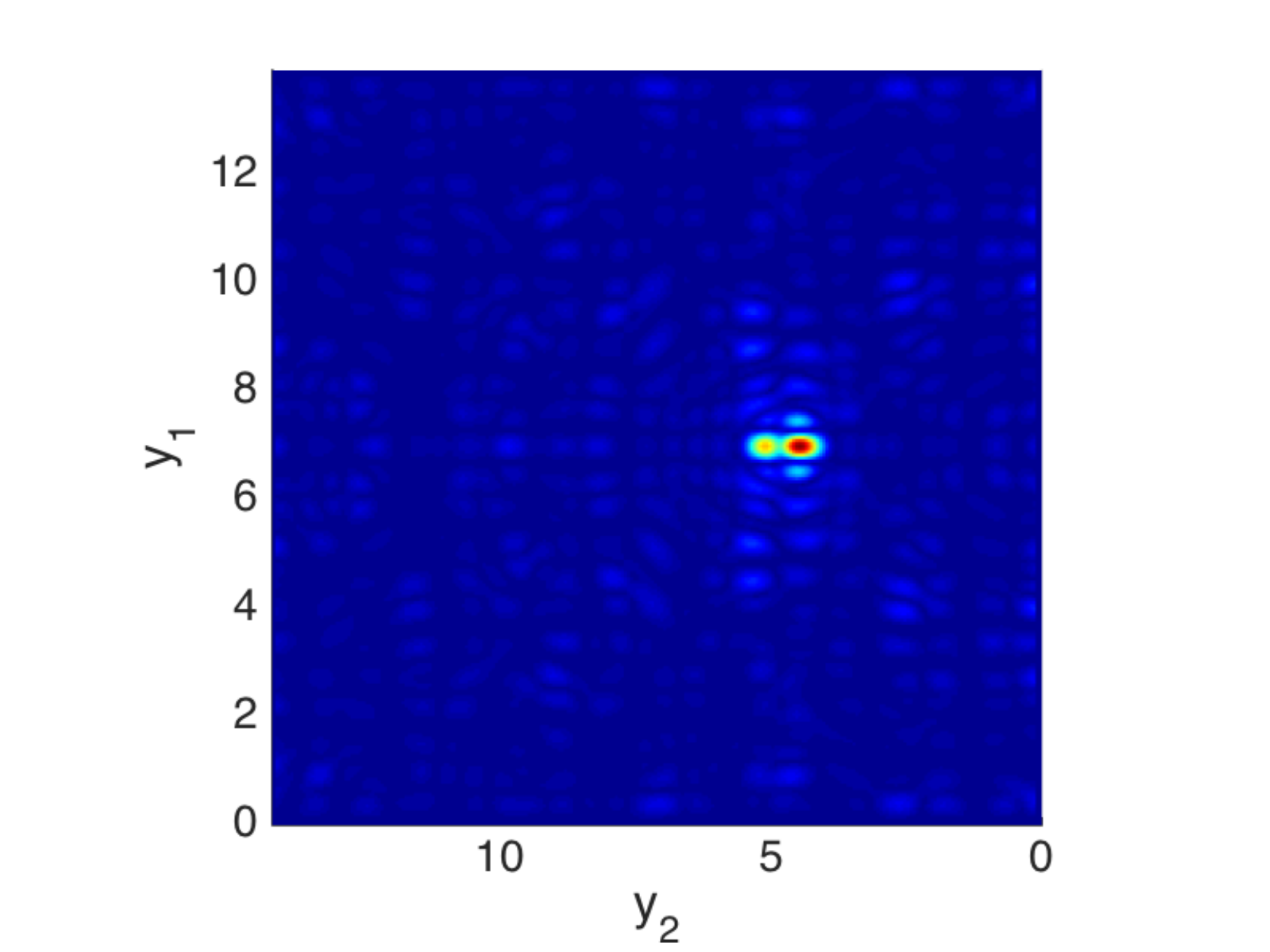}}
\caption{Reverse time migration images of a  point-like reflector located at   $(6.95,4.73,-10.44)\lambda$. The images in the first two rows are with $75\%$ aperture and those in the last row with the full 
aperture. The first row is for $100$ modes and the other two rows for $350$ modes. We display in the 
left column the images in the plane  $y_1 = 6.95\lambda$, and in the right column 
the images in the cross-range plane $y_3 = -10.44 \lambda$.  The axes are in units of $\lambda$.}
 \label{fi:2}
\end{figure}

In figure \ref{fi:2} we display the reverse time migration image of a point-like reflector
located at $(6.95,4.73,-10.44)\lambda$,  modeled by an isotropic scattering potential $V = v(\vby)I$ supported 
on a mesh cell in the imaging region. The mesh size is  $\lambda/18$ in cross-range plane and  $\lambda/6$ in range. We note that the reflector
 is well localized in range and cross-range, and the results improve, as expected when more modes 
 are used to form the image. Moreover, the image at $75\%$ aperture is almost as good as that with full 
 aperture. Naturally, the image deteriorates for smaller apertures.
 
The images of the same reflector obtained with $\ell_1$ optimization are shown in Figure
\ref{fi:4p}. They are obtained with the first $350$  arriving modes and a $75\%$ aperture. The discretization of these images is in steps of  $0.29\lambda$ in cross-range and $ 0.87 \lambda$ in range. As expected, 
these images give a sharper estimate of the support, in the sense that the spurious faint peaks 
in Figure \ref{fi:2} are suppressed in Figure \ref{fi:4p} by the sparsity promoting optimization.
  \begin{figure}[h!!t!!!b!!!!]
\centering
\hspace{-0.3in}\includegraphics[width=5.5cm]{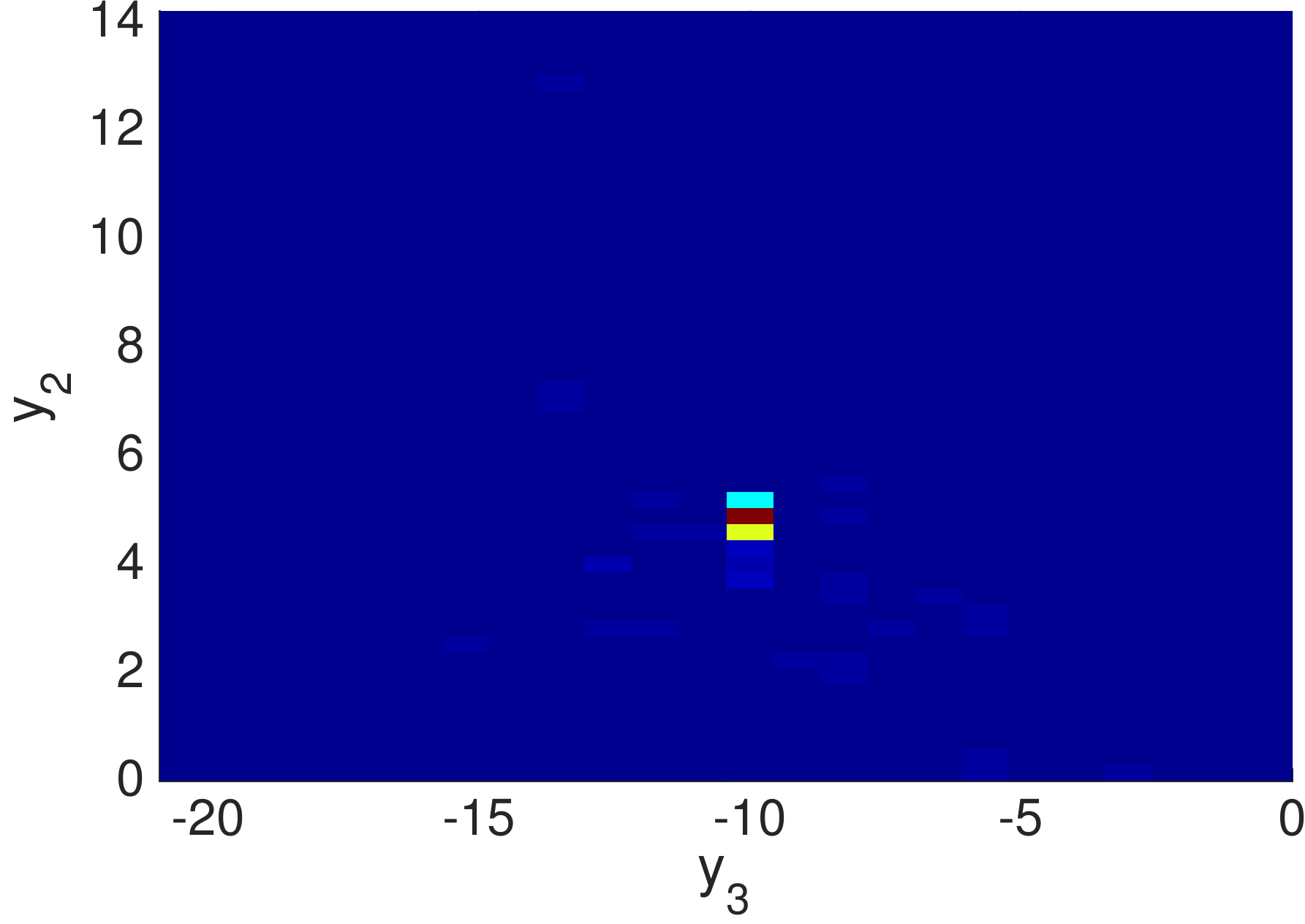}
\hspace{1cm}
\includegraphics[width=4.cm]{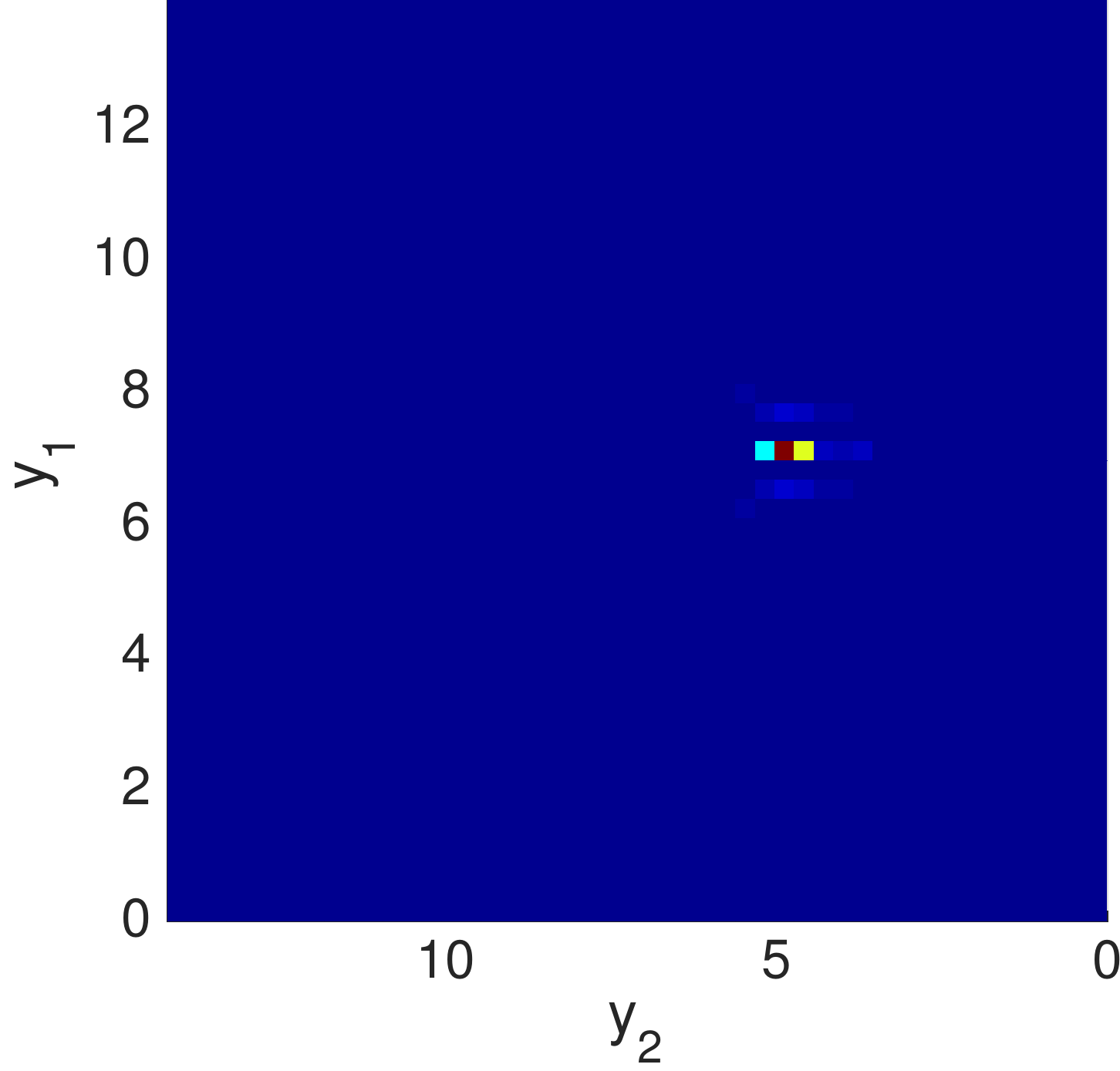}
\caption{Reconstructions of the same reflector as in Figure \ref{fi:2}  using $\ell_1$ optimization,
$350$ modes and $75\%$ aperture. The axes are in units of $\lambda$.
}
\label{fi:4p}
\end{figure}

 \begin{figure}[h!!t!!!b!!!!]
\centering
{\includegraphics[width=6cm]{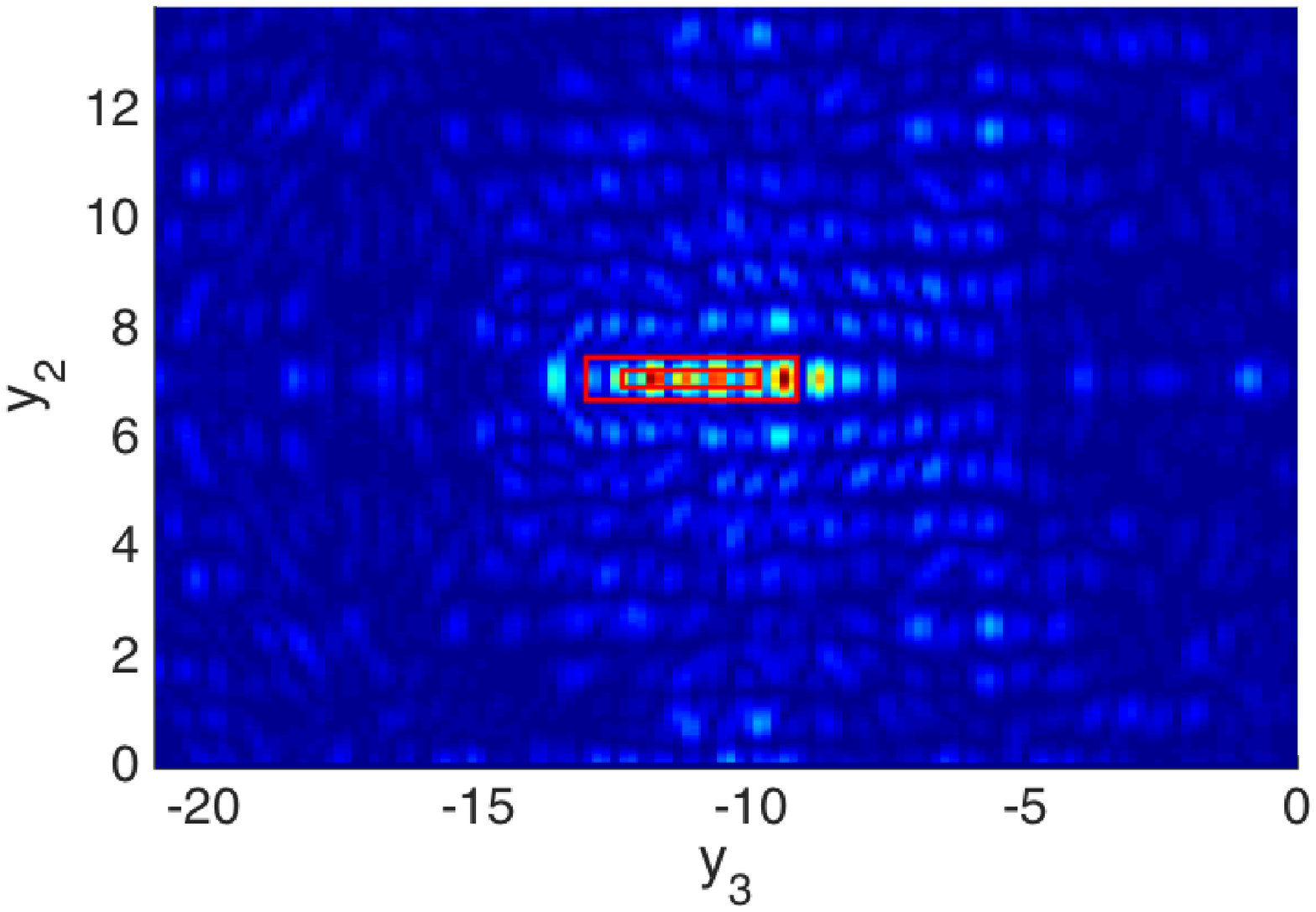}}
{\includegraphics[width=6cm]{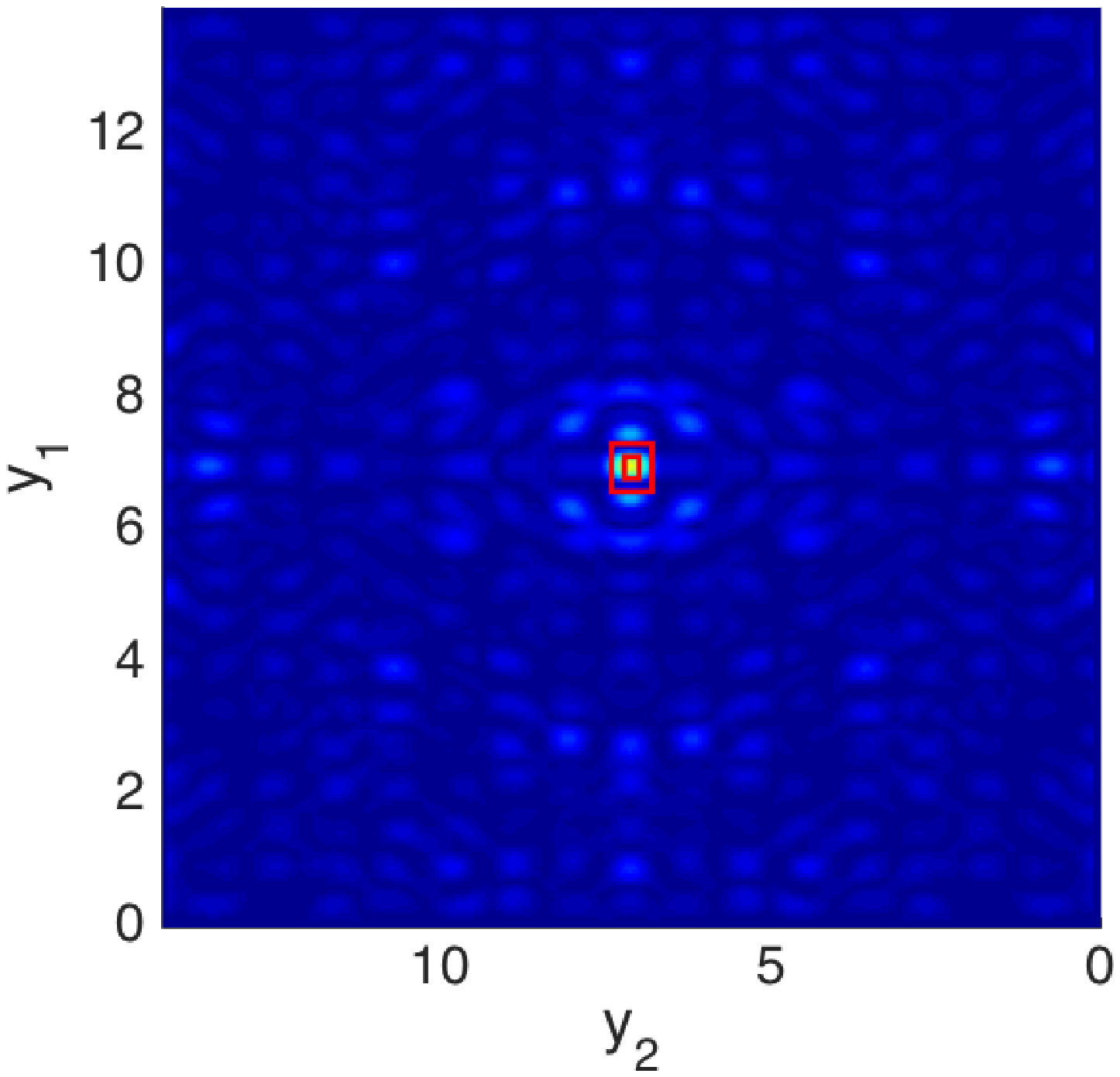}}\\
{\includegraphics[width=6cm]{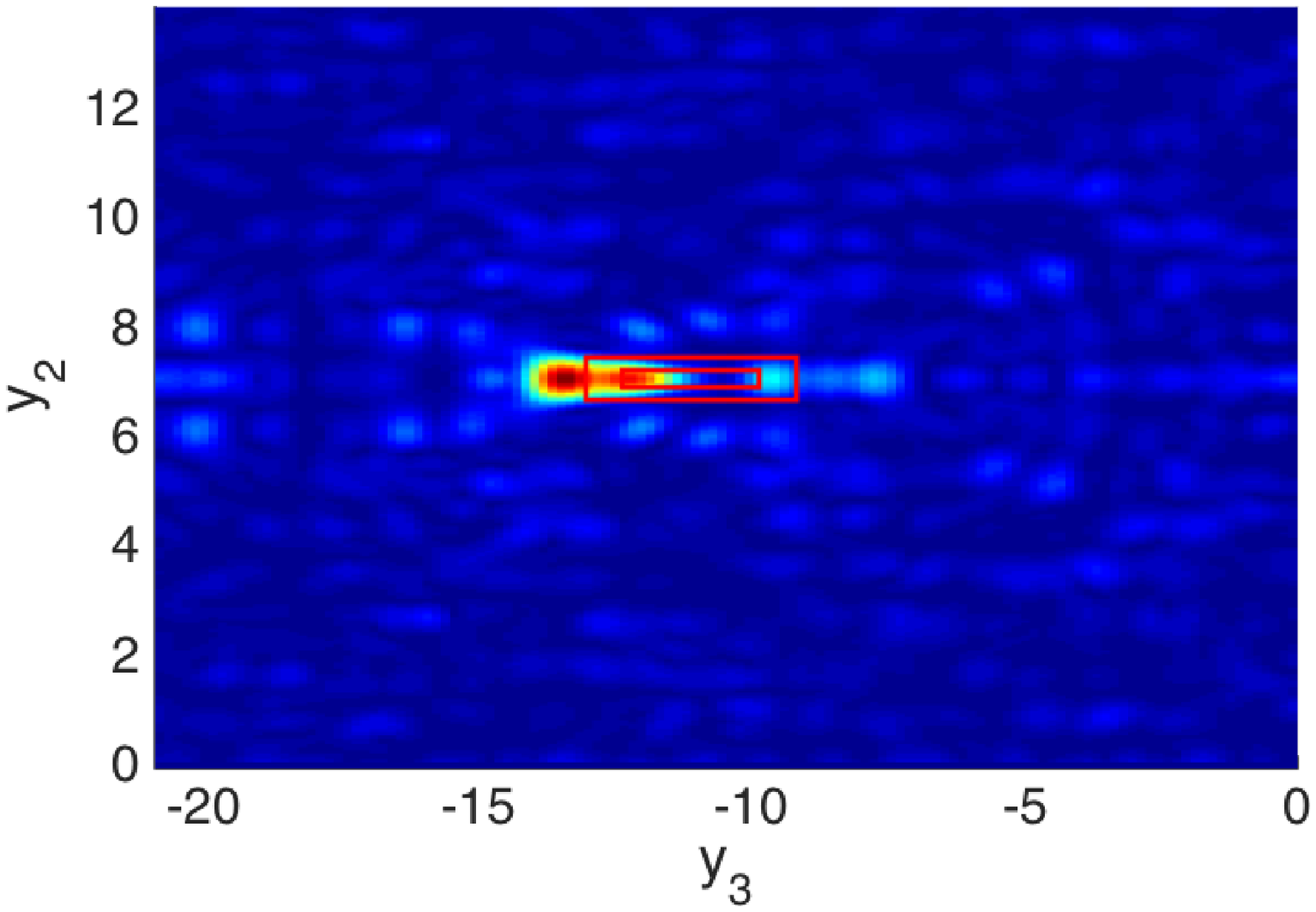}}
{\includegraphics[width=6cm]{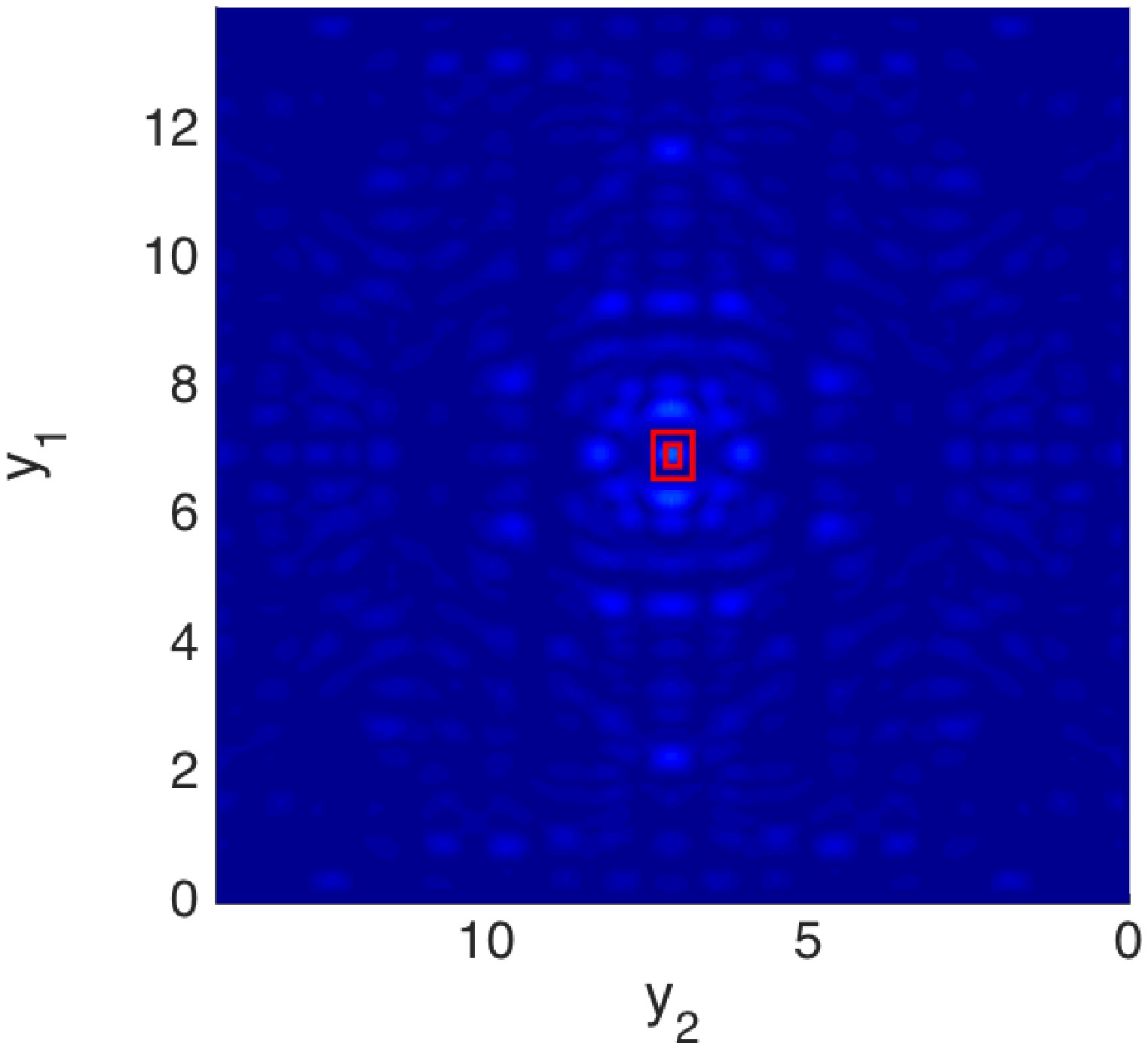}}\\
\caption{Reverse time migration images of a rectangular shell. The results in the first row are in the terminated waveguide. Those in the second row are in an infinite waveguide. We use the $350$ first arriving modes, $75\% $ aperture. The images in the left column are in the plane $y_1 = 6.96\lambda$ and in the right column 
in the cross-range plane at $y_3 = -11.14\lambda$.}
\label{fi:3}
\end{figure}

In Figures \ref{fi:3} and \ref{fi:5} we show images of an extended reflector shaped like 
a rectangular shell of \textcolor{black}{sides equal to $1.16 \lambda, 1.18\lambda$ and $3.9 \lambda$. 
The shell is modeled as $R\setminus R_o$, where $R = (6.38\lambda,7.54\lambda)\times(6.51\lambda,7.69\lambda)\times(-13.09\lambda,-9.19\lambda)$ and $R_o = (6.09\lambda,7.25\lambda)\times(6.22\lambda,7.40\lambda)\times(-12.22\lambda,-8.32\lambda)$.
The scattering potential equals 8 in $R\setminus R_o$  and zero outside, while the thickness of the cross-range  and range walls 
are $ 0.87 \lambda$ and $0.29\lambda$, respectively.}
The discretization of the imaging window  in Figure \ref{fi:3} is the same as in Figure \ref{fi:2}. We note that the reverse time migration method 
estimates better the support of the reflector, specially its back, in the terminating waveguide (top left image) than 
in the infinite waveguide (bottom left image). The  $\ell_1$ optimization images are in Figure  \ref{fi:5}, 
where the discretization of the imaging window is the same as in Figure \ref{fi:4p}.

 \begin{figure}[h!!t!!!b!!!!]
\centering
{\includegraphics[width=6cm]{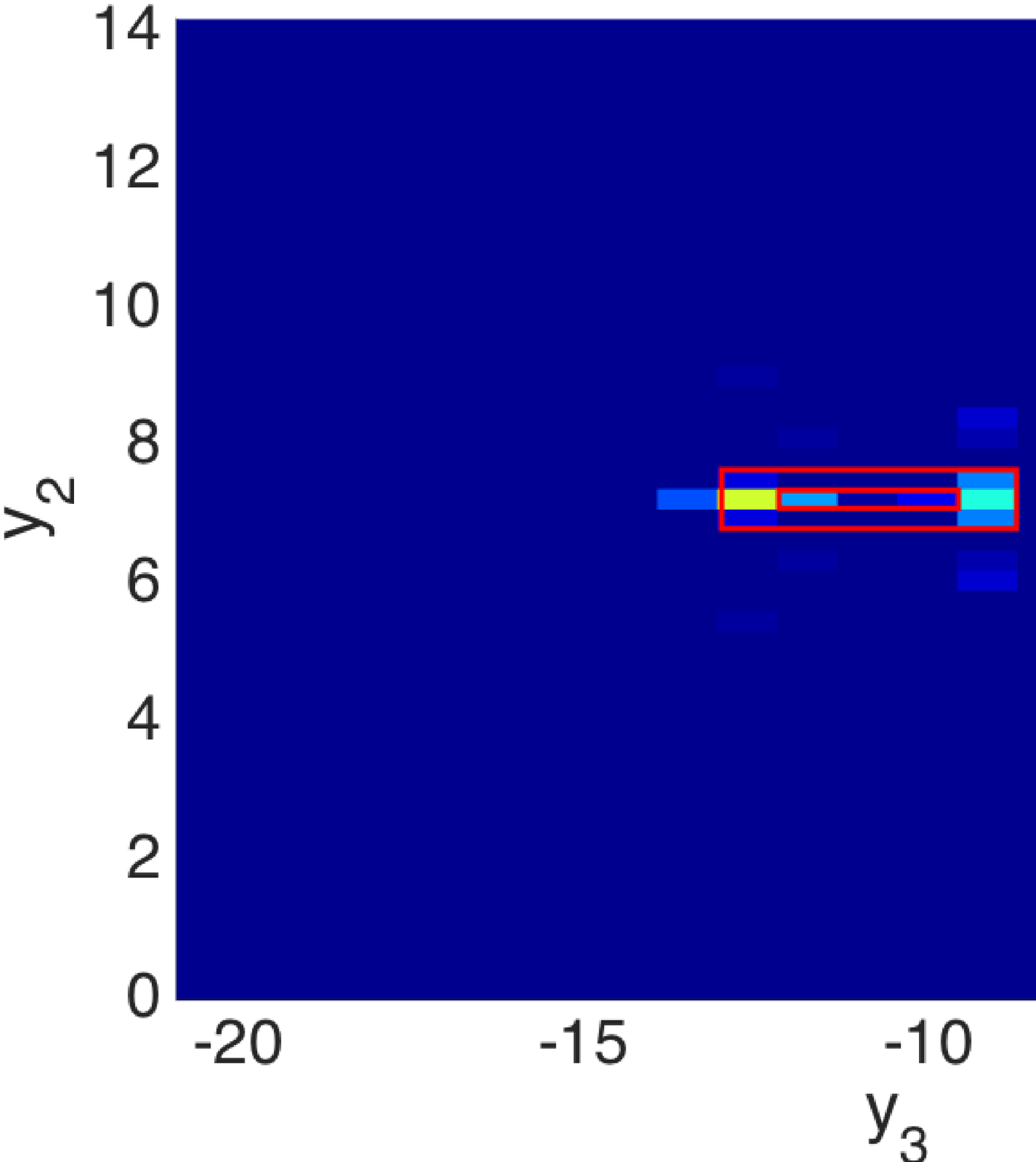}} 
{\includegraphics[width=6cm]{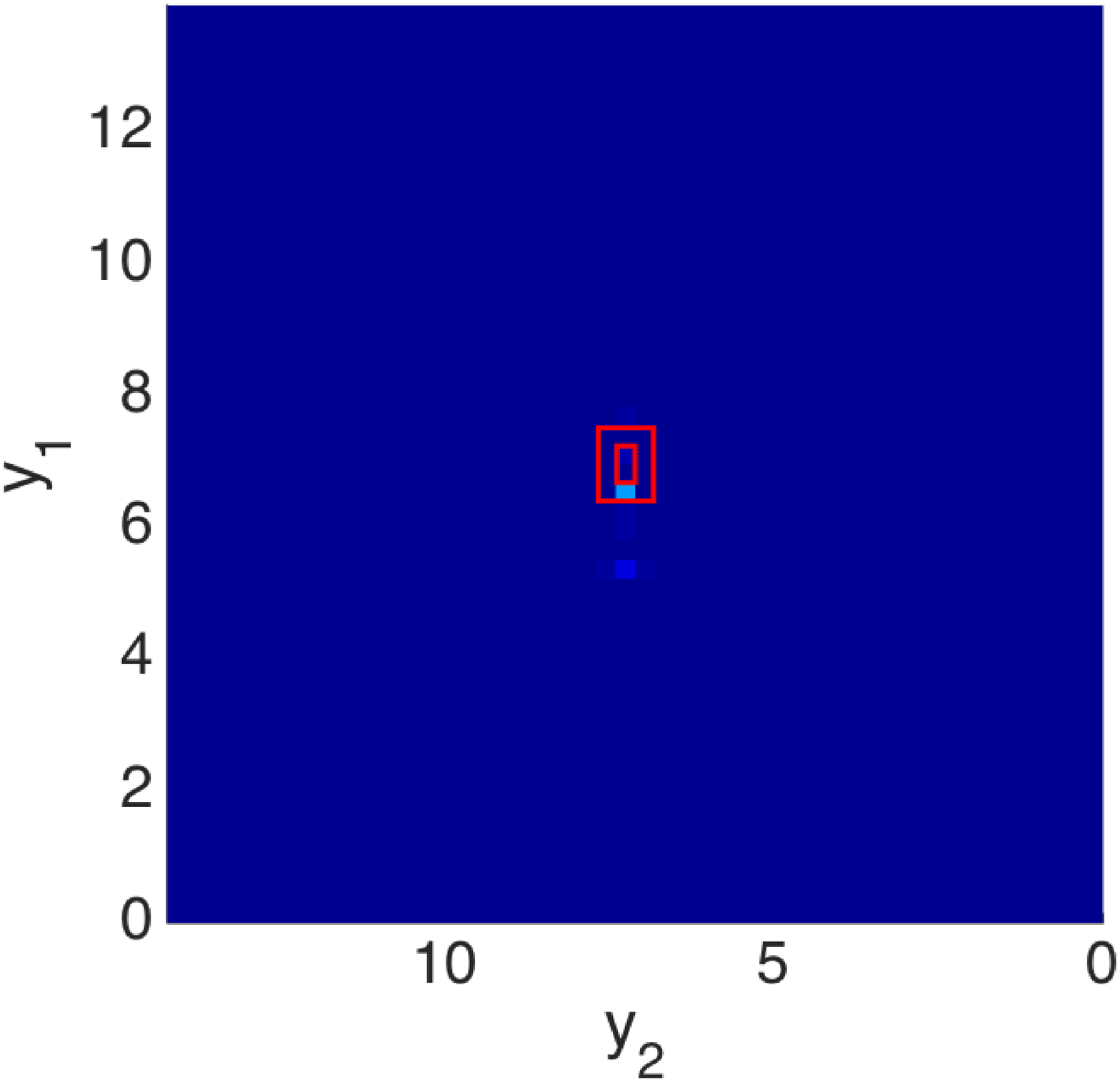}}\\
{\includegraphics[width=6cm]{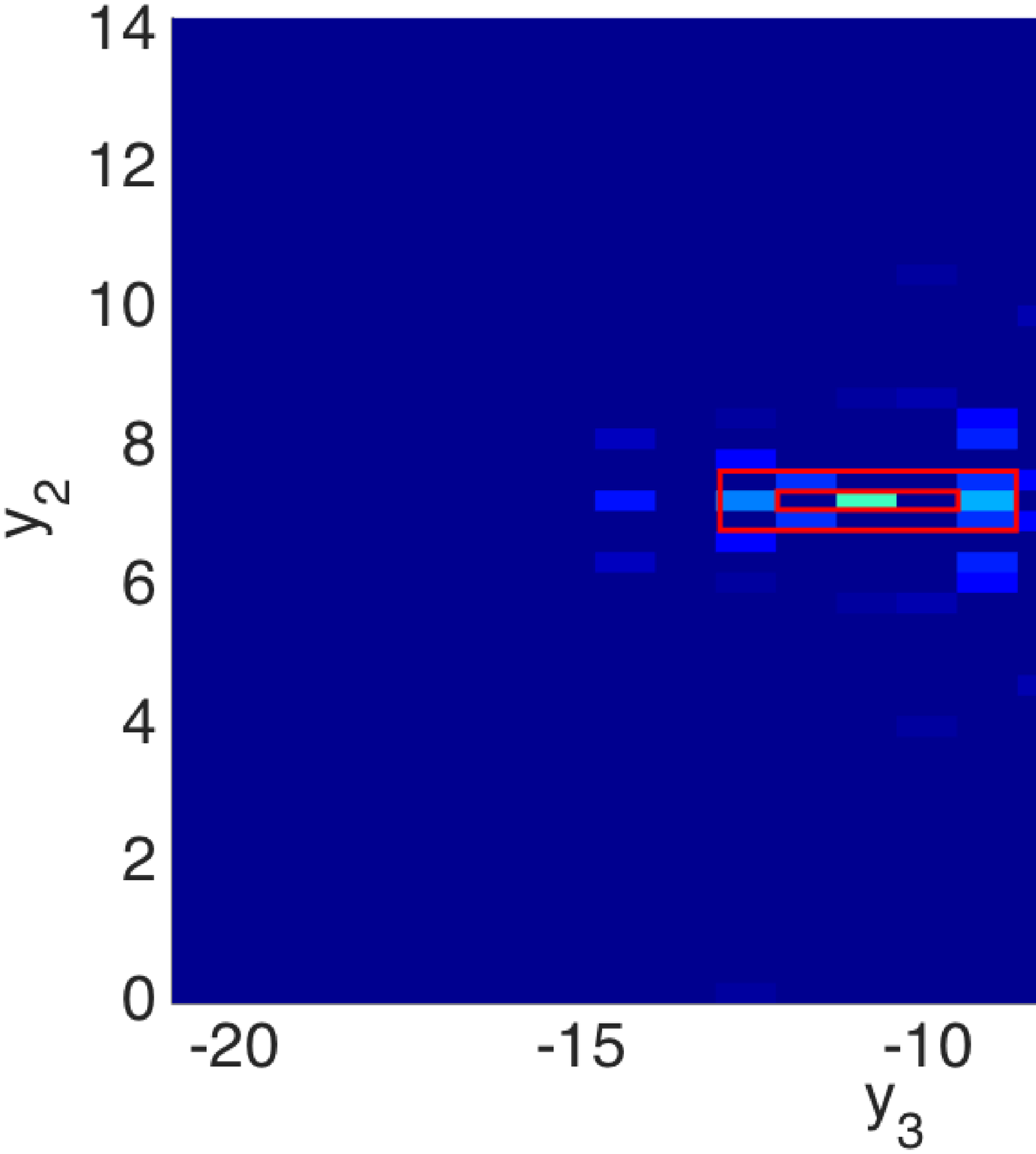}} 
{\includegraphics[width=6cm]{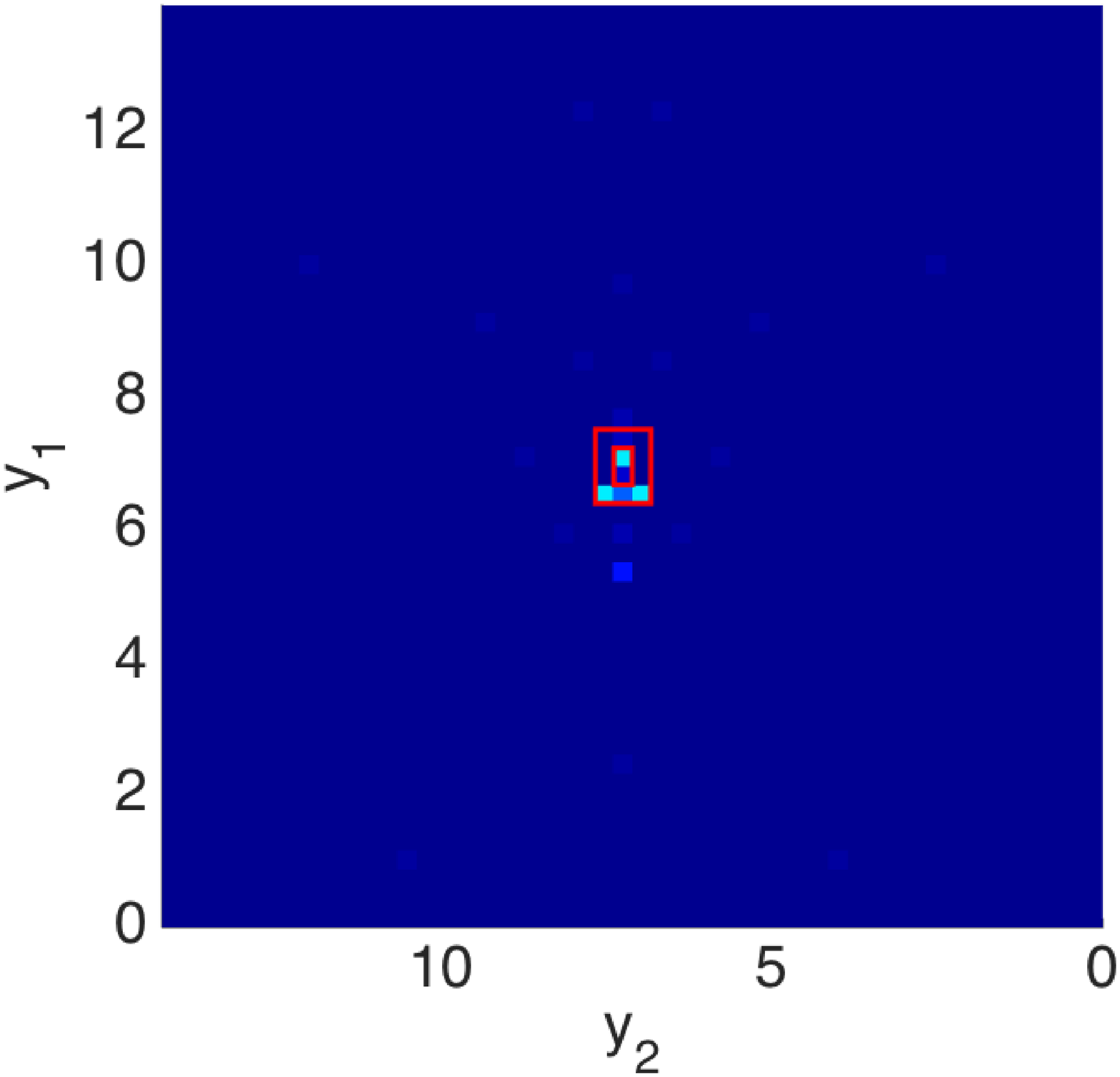}} 
\caption{Reconstructions of the same rectangle shell as in Figure \ref{fi:3}, using $\ell_1$ optimization, 350  modes  and 75\% aperture. 
In the top line we show the images in the terminating waveguide  and in the bottom line those 
in the  infinite waveguide. The images in the left column are in the plane $y_1 = 6.96\lambda$ and in the right column 
in the cross-range plane at $y_3 = -11.14\lambda$.}
\label{fi:5}
\end{figure}

In the last simulations in Figure \ref{fi:4} we present the images of an anisotropic point-like 
reflector, whose scattering potential is a diagonal matrix 
$V(\vby) = \mbox{diag} \big(3,1,5)v(\vby), $ with the same $v(\vby)$ as in Figure \ref{fi:2}.
We present only reverse time migration images and note that the estimates of the support of the
components of $V(\vby)$ are similar to those in Figure \ref{fi:2}. Specifically, we plot the 
absolute value of the right hand side of equation (\ref{eq:TR5}) for $l = 1, 2, 3$.

 \begin{figure}[h!!t!!!b!!!!]
\centering
{\includegraphics[width=6cm]{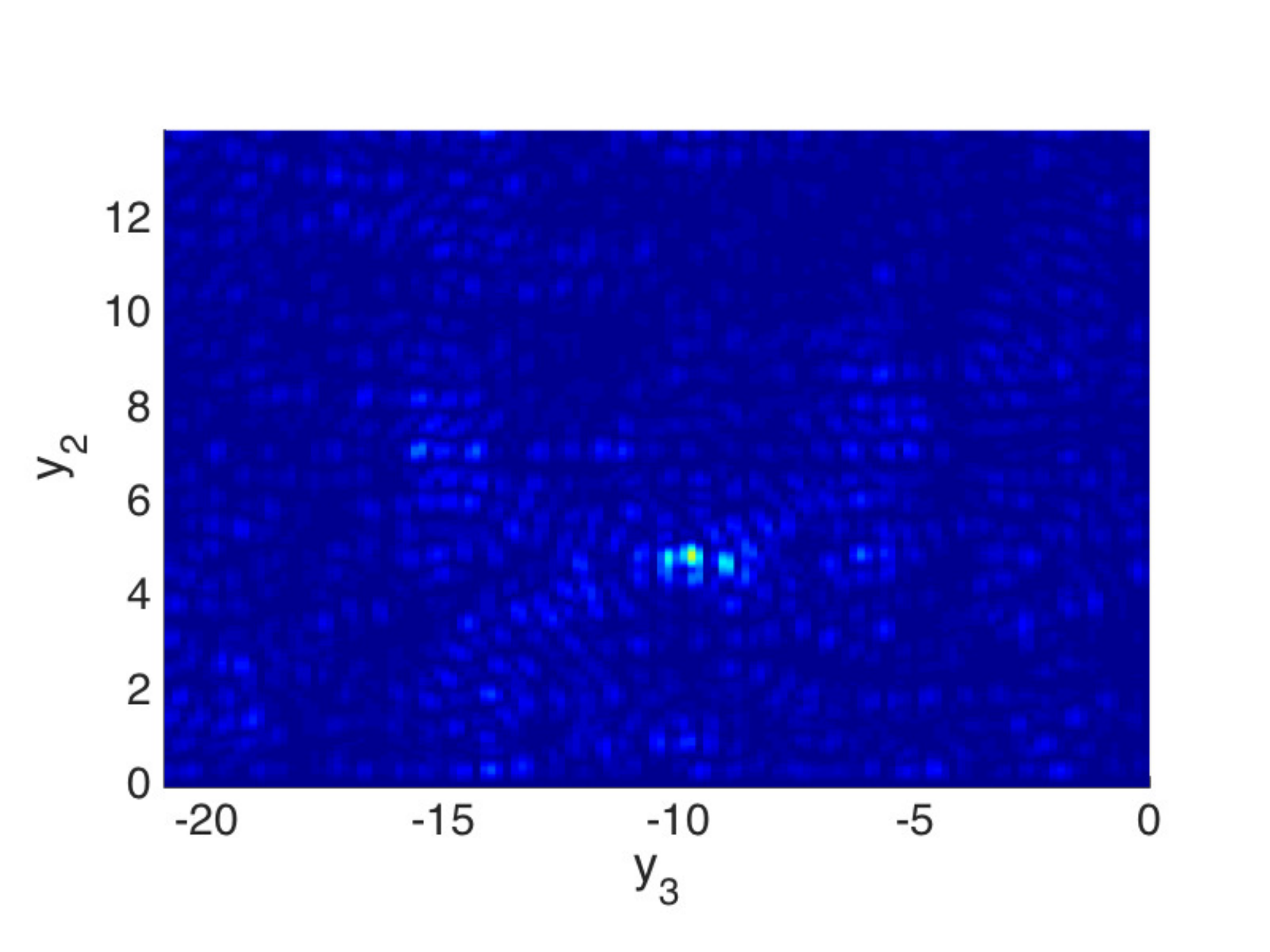}}
{\includegraphics[width=6cm]{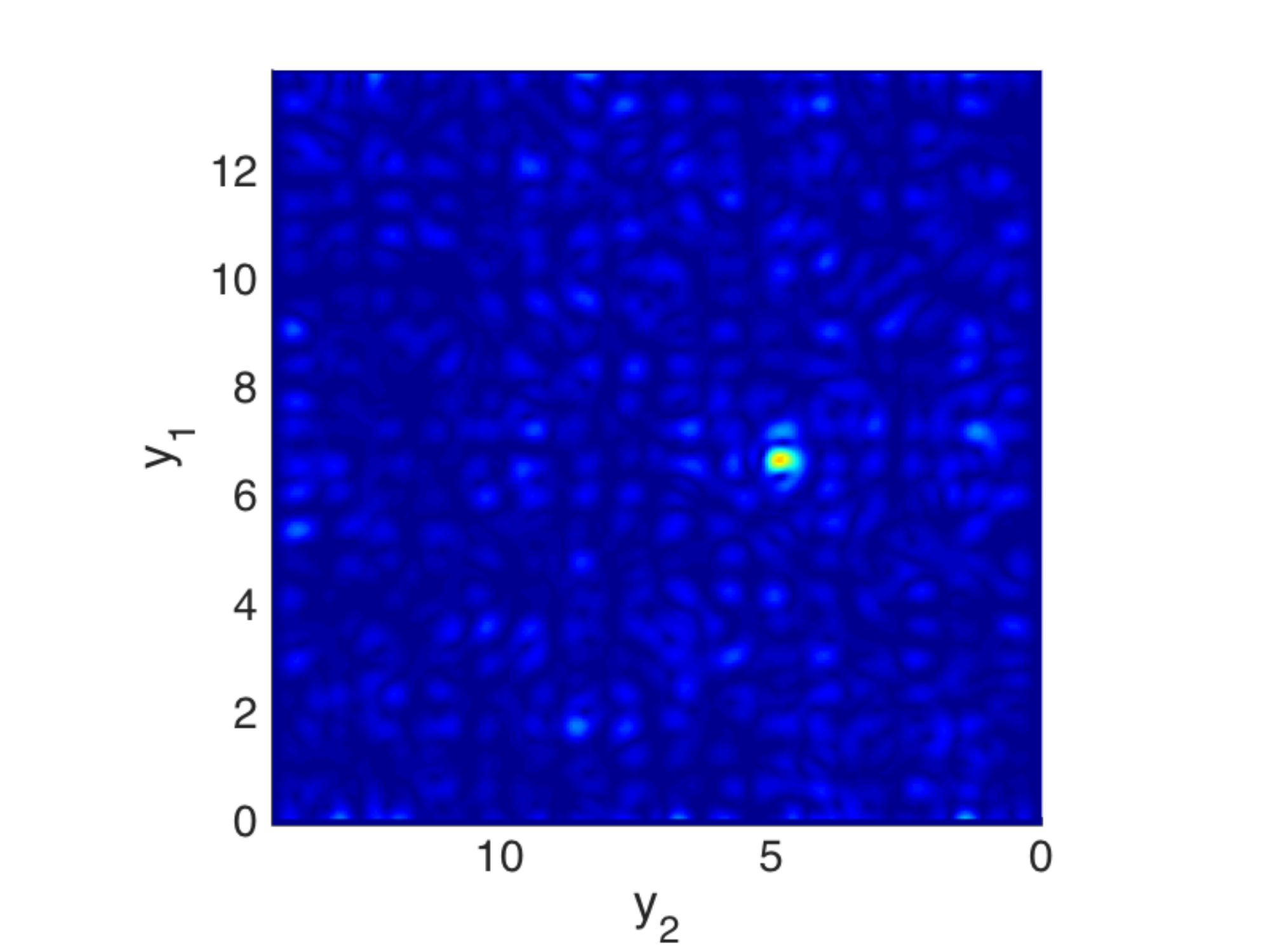}}\\
{\includegraphics[width=6cm]{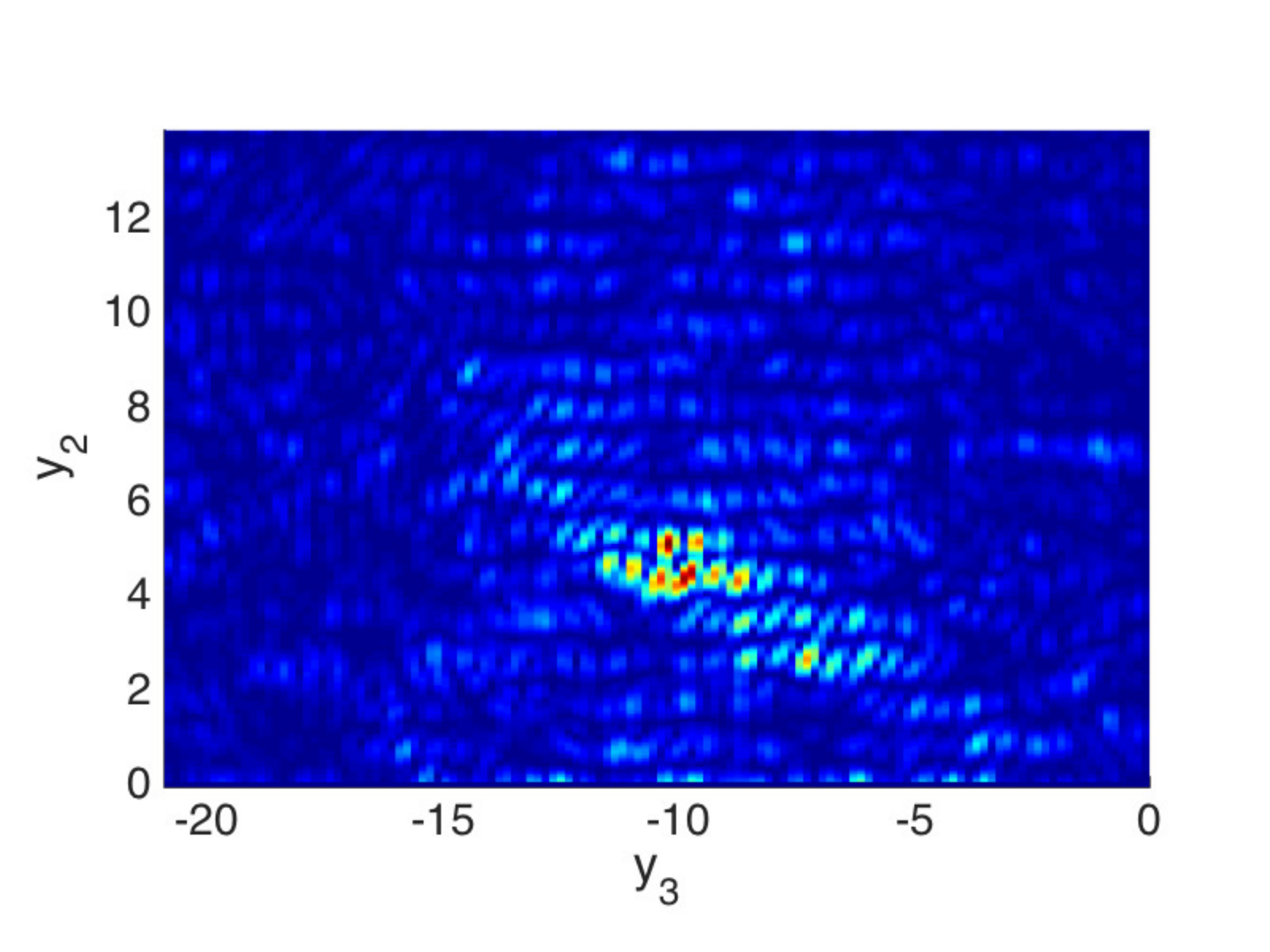}}
{\includegraphics[width=6cm]{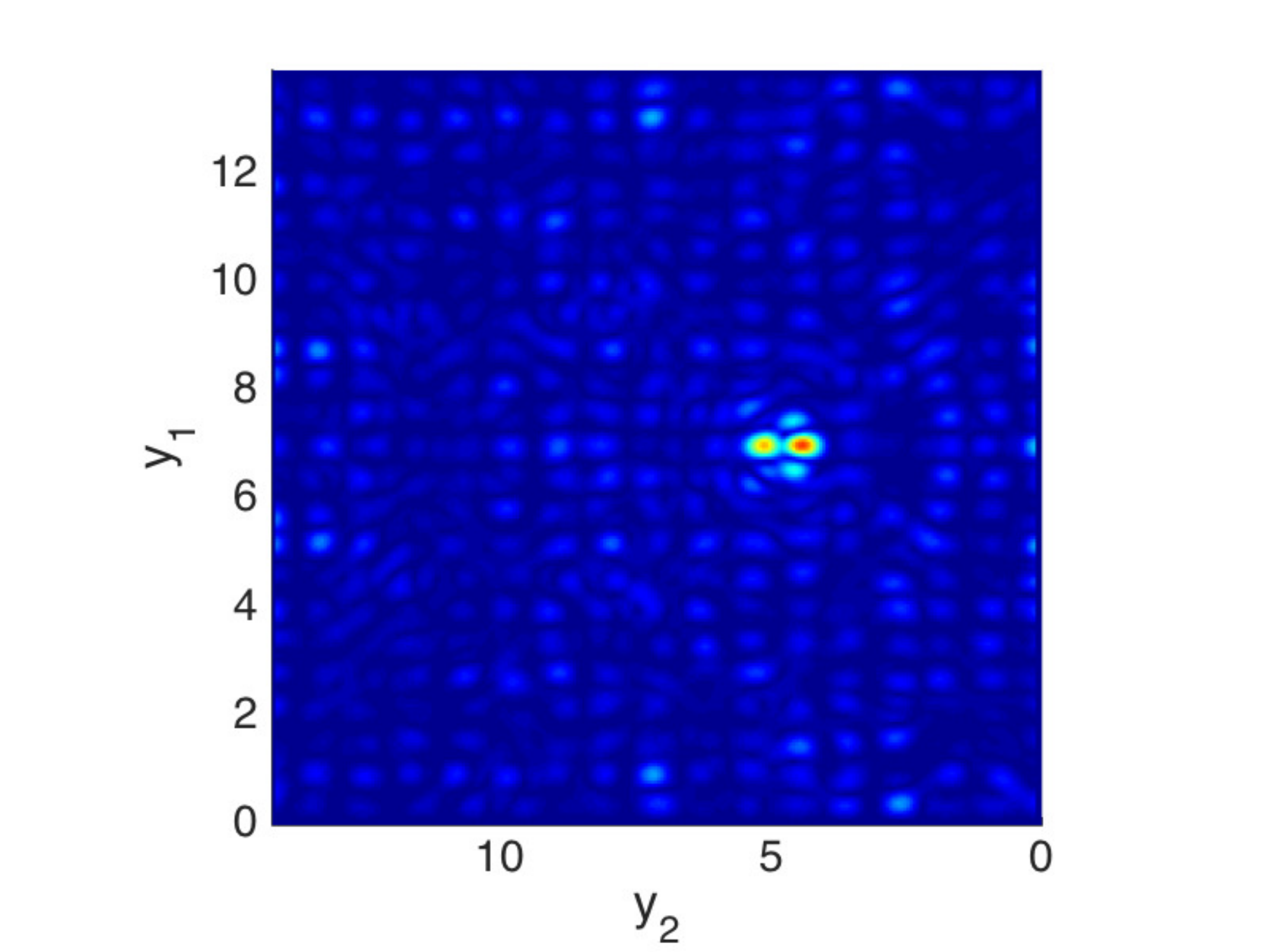}}\\
{\includegraphics[width=6cm]{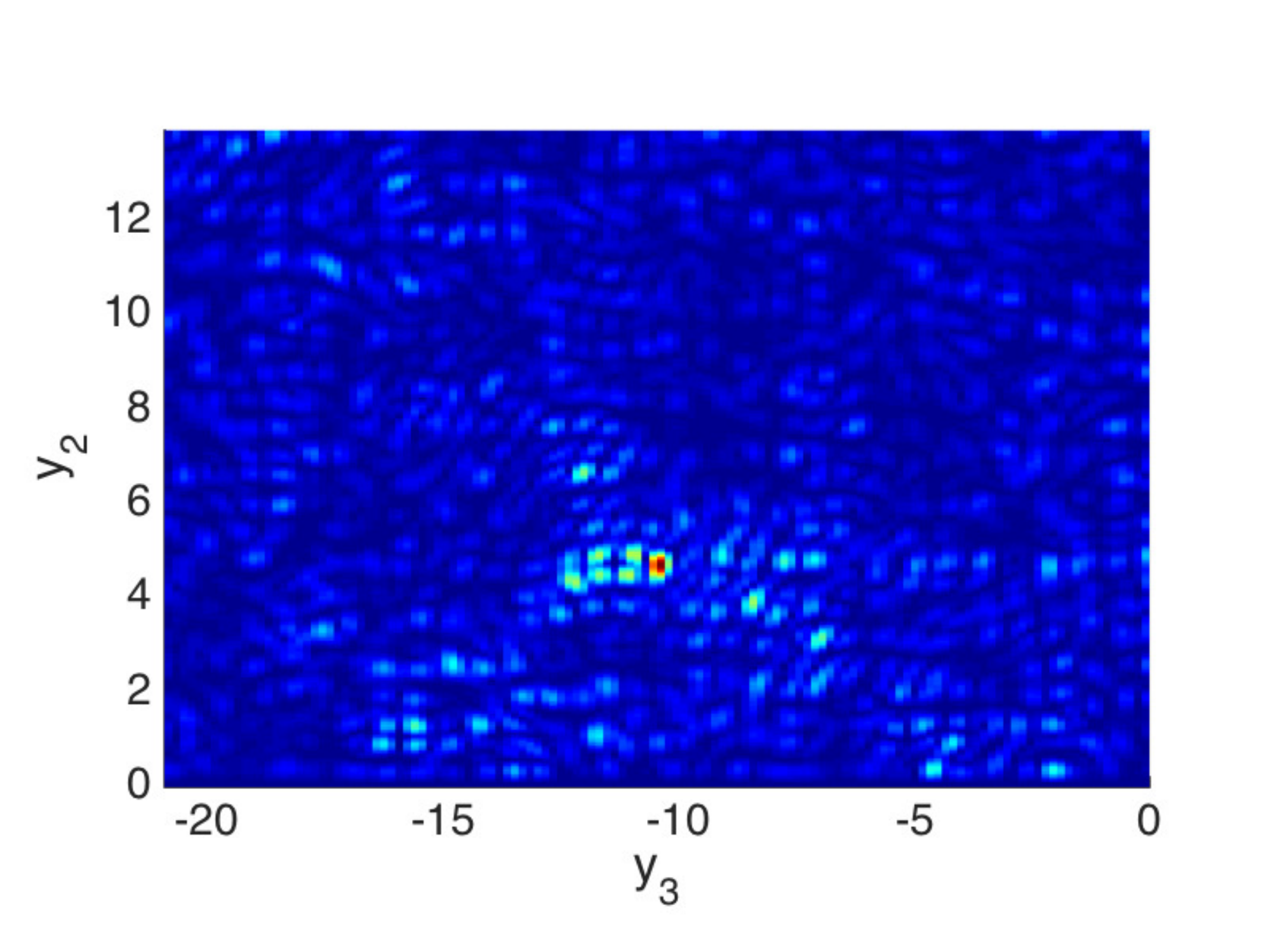}}
{\includegraphics[width=6cm]{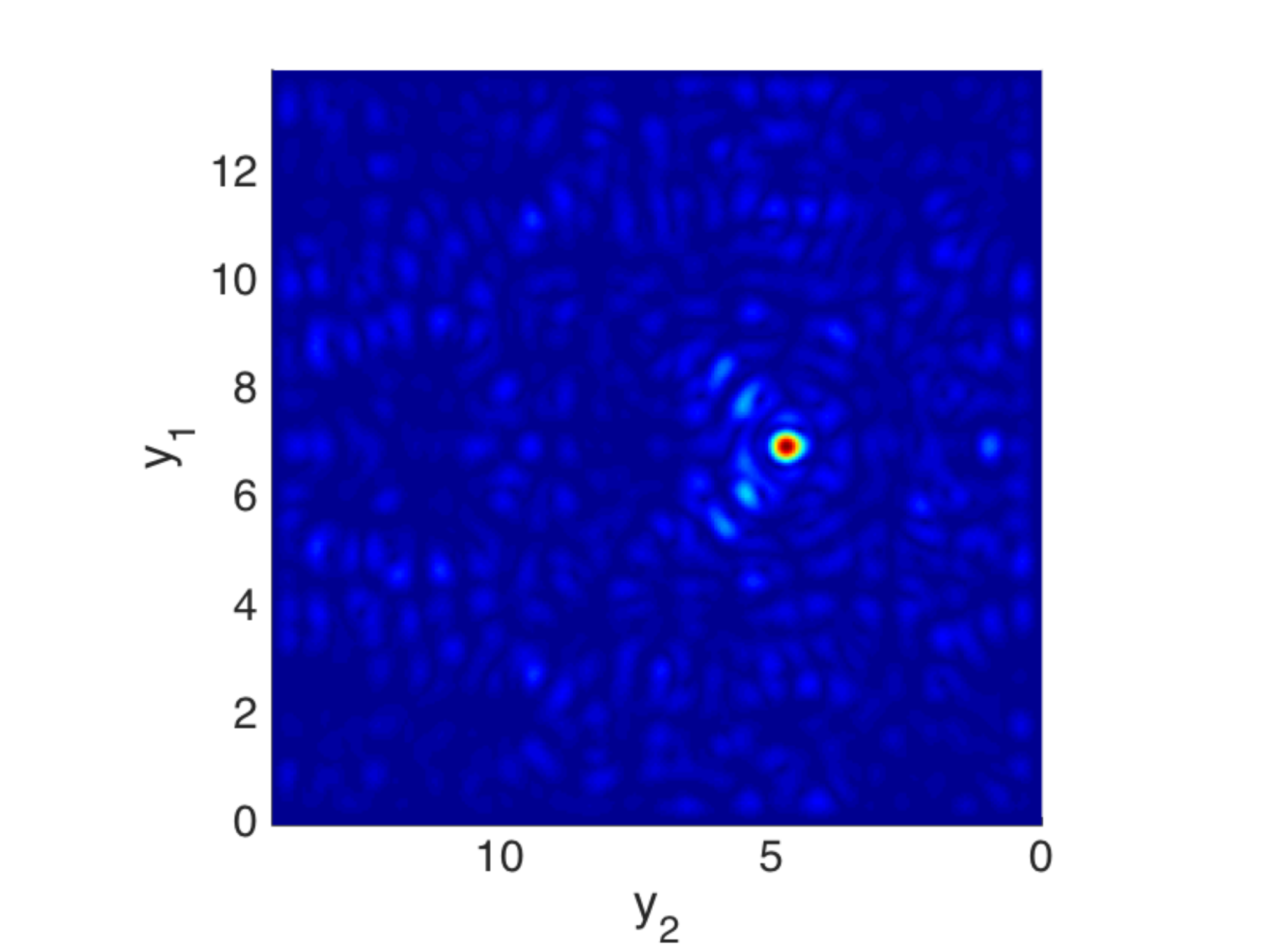}}\\
\caption{Reverse time migration images of an anisotropic point-like reflector located at $(6.95,4.73,-10.44)\lambda$, using reverse time migration. 
We use the first 350 arriving modes and 75\% aperture. We display the absolute value of (\ref{eq:TR5})
for $l = 1$ in the first line, $2$ in the second and $3$ in the third. 
The images in the left column are in the plane $y_1 = 6.95\lambda$ and in the 
right column in the plane   $y_3 = -10.44\lambda$. The axes are in units of $\lambda$.}
\label{fi:4}
\end{figure}

\section{Summary}
\label{sect:sum}
We study imaging  with electromagnetic waves in terminating waveguides,
using measurements of the electric field at an array of sensors. The goal of imaging is to 
localize compactly supported reflectors that lie between the array and the end wall. We derive  
the data model using Maxwell's equations. We define the scattered electric field 
due to an incident wave from one sensor in the array and show that it satisfies a Lipmann-Schwinger 
type equation.  We analyze the solvability of this equation and write explicitly the data model using 
a modal decomposition of the wave field in the waveguide. This model is based on the single scattering 
approximation at the unknown reflectors. We use it to formulate two imaging methods: The first 
forms an image by calculating the action of the  adjoint of the forward operator  on the data. It has a 
time reversal interpretation. The second uses $\ell_1$ i.e., sparsity enhancing optimization. We 
present numerical results with both imaging methods for point-like and extended reflectors.

 \section*{Acknowledgements}
This work was partially supported by AFOSR GrantFA9550-12-1-0117 (DLN)
and AFOSR grant FA9550-15-1-0118 (LB). LB also acknowledges support
from the Simons Foundation and ONR Grant N000141410077.

\appendix
\section{Vectorial eigenvalue problem} 
\label{sect:VEP}
\subsection{Spectral decomposition of the Laplacian}
Let $\vbf = (\bm f, f_3)^\top\in (L^2(\Omega))^3$, and consider the
linear differential operator associated with the vectorial Laplacian
problem
\begin{align}
-\Delta \vec{\bm{u}}(\bx) = \vbf(\bx) \quad &\bx \in \Omega, \nonumber 
\\ \bm{n}^\bot(\bx) \cdot \bm{u}(\bx) = \nabla\cdot \bm{u}(\bx) = 0
\quad &\bx \in \pa\Omega, \\ u_3(\bx) = 0 \quad &\bx \in \pa\Omega,
\label{eq:A1}
\end{align}
for $\vec{\bm{u}} = (\bm{u},u_3)$.  Since $\Delta \vec{\bm{u}} =
(\Delta \bm{u}, \Delta u_3)^\top $, we have two decoupled
problems. One is the standard Poisson problem for the longitudinal
component $u_3$,
\begin{align}
-\Delta u_3(\bx) &= f_3(\bx) \quad \bx \in \Omega, 
\nonumber \\
u_3(\bx) & =0 \quad \bx \in \pa \Omega,\label{eq:u3}
\end{align}
whose weak solution is in $H^1_0(\Omega)$ and satisfies
\begin{equation}
b(u_3,v) = \int_\Omega \nabla u_3(\bx)\cdot \nabla \ol{v}(\bx) \,\d
\bx = (f_3,v)_{L^2}, \quad \text{for all } v\in \bm{H}^1_{0}(\Omega),
\label{eq:weak01}
\end{equation} 
where $(\cdot,\cdot)_{L^2}$ denotes the inner product in
$L^2(\Omega)$.  The other problem is for the two dimensional
transverse vector $\bm{u}$,
\begin{align}
-\Delta \bm u(\bx) &= \bm f(\bx) \quad  \bx \in \Omega, \nonumber \\
\bm{n}^\bot(\bx) \cdot \bm{u}(\bx)  &= 0 \quad  \bx \in \pa\Omega, \nonumber \\
\nabla \cdot \bm u(\bx) &= 0 \quad  \bx \in  \pa\Omega. \label{eq:ub}
\end{align}
It is studied in \cite{Kangr1999} for a more general $\Omega$ than the
rectangle considered here. The results there establish the existence
and uniqueness of weak solutions in the space
\[
\bm{H}^1_{0t}(\Omega) = \{\bm u \in \big(H^1(\Omega)\big)^2: \bm{n}^\bot
\cdot \bm{u} =0 \text{ on } \pa\Omega \},
\]
with the standard inner $H^1$ product $(\bm u,\bm v)_{H^1}$. These
solutions satisfy the variational problem
\begin{align}
\label{eq:weak1}
a(\bm u,\bm v) = \int_\Omega (\nabla^\perp\cdot \bm u (\bx) \,
\nabla^\perp \cdot \ol{\bm v}(\bx)+ \nabla\cdot\bm u(\bx)
\,\nabla\cdot \ol{\bm v}(\bx)) \d \bx = (\bm f, \bm v)_{L^2}, 
\end{align}
for all $\bm v\in \bm{H}^1_{0t}(\Omega)$, where $\nabla^\perp$ is the
rotated gradient operator, playing the role of curl in two-dimensions,
and $(\cdot,\cdot)_{L^2}$ is the inner product in
$\big(L^2(\Omega)\big)^2$.  The results in \cite{Kangr1999} also give
a proper interpretation of $\nabla \cdot \bm u|_{\pa \Omega}$ in terms
of the curvature of the boundary. In our case the boundary is the
union of four line segments $\partial \Omega_j$, for $j = 1, \ldots,
4$, so the curvature is zero on each segment.  It is shown in
\cite{Kangr1999} that $\nabla \cdot \bm u|_{\pa \Omega}$ exists and
belongs to $H^{-1/2}(\pa \Omega_j)$ on each piece of the boundary, and
the weak solution satisfies the estimate
\begin{align}
\label{eq:estimate}
\|\bm u\|_{H^1} \leq C\|\bm f\|_{L^2}.
\end{align}

To arrive at the spectral decomposition of the vectorial Laplacian in
(\ref{eq:A1}), we study its ``inverse'' i.e., the solution operator
$\mathcal{L}: (L^2(\Omega))^3 \to (L^2(\Omega))^3$ defined by
$\mathcal{L}(\vec{\bm f}) = (\bm u, u_3)$, where $\bm u$ solves
(\ref{eq:weak1}) and $u_3$ solves (\ref{eq:weak01}). Obviously,
$\mathcal{L}$ is a linear operator.  It is also injective, bounded,
self-adjoint and compact. The injectivity follows from the uniqueness
of solutions of (\ref{eq:weak01}) and (\ref{eq:weak1}). The
boundedness and compactness follow from the estimate
(\ref{eq:estimate}) on $\bm u$ and a similar one on $u_3$, together
with the imbedding of $H_{0t}^1(\Omega)$ in $(L^2(\Omega))^2$ and of
$H_0^1(\Omega)$ in $L^2(\Omega)$. To see that $\mathcal L$ is
self-adjoint, let $\vbf$ and $\vbg$ be arbitrary in $(L^2(\Omega))^3$
and denote by $(\bm u,u_3)$ and $(\bm v,v_3)$ their image in
$\Im(\mathcal{L}) \subset (L^2(\Omega))^3$, such that
$\mathcal{L}(\vbf) = (\bm u,u_3)$ and $\mathcal{L}(\vbg) = (\bm
v,v_3)$. Then equations (\ref{eq:weak01}) and (\ref{eq:weak1}) give
 \begin{align*}
 (\mathcal{L}(\vbf),\vbg)_{L^2} &= (\bm u, \bm g)_{L^2} + (u_3,
   g_3)_{L^2} = \ol{a(\bm v, \bm u)} + \ol{b(v_3,u_3)}\\ &= a(\bm u,
   \bm v) + b(u_3,v_3) = (\bm f, \bm v)_{L^2} +(f_3,v_3)_{L^2}\\ &=
   (\vbf, \mathcal{L}(\vbg))_{L^2}.
 \end{align*}
We conclude from the spectral theorem for self-adjoint, compact operators
\cite[appendix D]{evans} that there is an orthogonal basis of
$(L^2(\Omega))^3$ consisting of the eigenfunctions $\vec{\bm u}_j$ of
$\mathcal{L}$, for eigenvalues $\gamma_j$ that tend to $0$ as $j \to
\infty$. 

The eigenvalues cannot be zero, because $\mathcal{L}$ is injective, so
we can divide by them and get
\[
\vec{\bm u}_j = \gamma_j^{-1} \mathcal{L}(\vec{\bm u}_j).
\]
Consequently, by estimate (\ref{eq:estimate}) and a similar one for
the standard problem (\ref{eq:weak01}), we obtain that $\vec{\bm u}_j
= (\bm u_j,u_{3,j}) \in \mathcal{H}$, the space of three dimensional
vectors with components $\bm u_j \in H_{0t}^1(\Omega)$ and $u_{3,j}
\in H_0^1(\Omega)$.  To finish the argument, let $\vec{v} = (\bm v,
v_3) \in \mathcal{H}$ and consider the bilinear form $A: \mathcal{H}
\times \mathcal{H} \to \mathbb{C}$ defined in the obvious way
\[
A(\vec{\bm u},\vec{\bm v}) = a(\bm u, \bm v) + b(u_3,v_3).
\]
By letting $\vec{\bm u} = \vec{\bm u}_j$ we get 
\begin{align*}
\gamma_j A(\vec{\bm u}_j,\vec{\bm v}) = A \big( \mathcal{L}(\vec{\bm
  u}_j),\vec{\bm v} \big) = ({\bm u}_j,\bm v)_{L^2} +
(u_{3_j},v_3)_{L^2} = \lin\vec{\bm u}_j,\vec{\bm v}\rin,
\end{align*}
where we used equations (\ref{eq:weak01}) and (\ref{eq:weak1}) and
recall that $\lin \cdot, \cdot \rin$ is the inner product in
$\big(L^2(\Omega)\big)^3$. This relation states that $\vec{\bm u}_j$ are weak
eigenfunctions of the vectorial Laplacian, for eigenvalues $\lambda_j
= \gamma_j^{-1}$.

Finally, the expression (\ref{eq:eigenvals}) of the eigenvalues and
(\ref{eq:eigvect1})--(\ref{eq:eigvect4}) of the eigenfunctions follow
by direct calculation i.e., the method of separation of variables. See
\cite[section 3]{alonso2015electromagnetic}.  The eigenfunctions
$\vec{\bm u}_j$ are denoted in the paper by $\vec \Phi_j^{(s)}(\bx)$,
with index $s = 1, \ldots, m_j,$ the multiplicity of the
eigenvalue $\lambda_j$.

\section{The reference field}
\label{sect:REF}
Because the eigenfunctions $\vec \Phi_j^{(s)}(\bx)$ are an orthogonal
basis, we can seek the solution $\vec{\bm E}^o$ of equations
(\ref{eq:forward4}) in the form 
\begin{align}
\label{eq:reference}
\vec{\bm{E}}^{o}(\vx) = \sum_{ j\in\N^2_0}\sum_{s=1}^{m_j}
g^{(s)}_{j}(x_3) \vec\Phi^{(s)}_j(\bx),
\end{align}
for each given $x_3< 0$. It remains to determine the coefficients
$g_{j}^{(s)}(x_3).$

We substitute (\ref{eq:reference}) in (\ref{eq:forward4}), and calculating
\begin{align}
\curl\curl\left[g^{(1)}_{n}(x_3)\Phi^{(1)}_j(\bm x)\right] &=
               [\lambda_ng^{(1)}_{n}(x_3) -
                 \pa_{x_3}g^{(1)}_{j}(x_3)]\vec\Phi^{(1)}_j(\bx),
               \nonumber \\
\label{eq:formula}
\curl\curl\left[g^{(2)}_{n}(x_3)\Phi^{(2)}_j(\bm x)\right] &=
-\pa^2_{x_3}g^{(2)}_{j}(x_3) \Phi^{(2)}_j(\bx)- \lambda_j
\pa_{x_3}g^{(2)}_{j}(x_3) \vec \Phi^{(3)}_j(\bx),
\\ \curl\curl\left[g^{(3)}_{n}(x_3)\Phi^{(2)}_j(\bm x)\right]&=
\lambda_j g^{(3)}_{j}(x_3) \Phi^{(3)}_j(\bx) + \pa_{x_3}
g^{(3)}_{j}(x_3)\Phi^{(2)}_j(\bx), \nonumber
\end{align}
we get 
\begin{align}
i \om \mu_o \vbJ(\bx) \delta(x_3+L) = \sum_{
  j\in\N^2_0}\sum_{s=1}^{m_j} \Big\{ \big[ (\lambda_j-k^2)
  g_j^{(1)}(x_3) - \pa_{x_3}^2 g_j^{(1)}(x_3) \big] \vec
\Phi_j^{(1)}(\bx) \delta_{s,1} \nonumber \\ +\big[ -(k^2 + \pa_{x_3}^2
  ) g_j^{(2)}(x_3) + \pa_{x_3}g_j^{(3)}(x_3) \big] \vec
\Phi_j^{(2)}(\bx) \delta_{s,2} \nonumber \\ + \big[ (\lambda_j - k^2)
  g_j^{(3)}(x_3) -\lambda_j \pa_{x_3}g_j^{(2)}(x_3) \big] \vec
\Phi_j^{(3)}(\bx) \delta_{s,3} \Big\}. \label{eq:NEW}
\end{align}
The equations for $g_j^{(s)}(x_3)$ follow from (\ref{eq:NEW}) and the
orthogonality of the eigenfunctions,
\begin{align*}
\pa_{x_3}^2 g^{(1)}_{j}(x_3) &= -(k^2-\lambda_j)g^{(1)}_{j}(x_3), \\
\pa_{x_3}^2 g^{(2)}_{j}(x_3) &= -(k^2-\lambda_j)g^{(2)}_{j}(x_3), 
\\ g^{(3)}_{j}(x_3) &= \frac{\lambda_j}{\lambda_j-k^2}\pa_{x_3} g^{(2)}_{j}(x_3), \quad 
x_3 \ne -L.
\end{align*}
The solution of these equations is
\begin{align}
\label{eq:g12}
g^{(s)}_{j}(x_3) = a^{\pm(s)}_{j} e^{i\beta_jx_3} + b^{\pm(s)}_{j}
e^{-i\beta_jx_3}, \quad \text{for } s=1,2, 
\end{align}
and 
\begin{align}
\label{eq:g3}
g^{(3)}_{j}(x_3) =
\frac{\lambda_j}{\lambda_j-k^2}\left[i\beta_ja^{\pm(s)}_{j}
  e^{i\beta_jx_3} -i\beta_j b^{\pm(s)}_{j} e^{-i\beta_jx_3}\right],
\end{align}
where $\pm$ stands for the right and left of source.  The amplitudes
$a_j^{\pm (s)}$ and $b^{\pm (s)}_j$ have the expression given in
(\ref{eq:ampA1})-(\ref{eq:ampB2}). They are derived from the jump
conditions at the source, 
\begin{align*}
-\left[\pa_{x_3}g^{(1)}_{j}\right]_{-L} &= \frac{i\omega\mu_o
  (\vec\Phi^{(1)}_j,\vbJ)}{\|\vec\Phi^{(1)}_j\|^2},\quad
\left[g^{(1)}_{j}\right]_{-L} =
0,\\ -\left[\pa_{x_3}g^{(2)}_{j}\right]_{-L} +
\left[g^{(3)}_{j}\right]_{-L} &=
\frac{i\omega\mu_o(\vec\Phi^{(2)}_j,\vbJ)}{\|\vec\Phi^{(2)}_j\|^2},
\\ -\lambda_j \left[g^{(2)}_{j}\right]_{-L} &=
\frac{i\omega\mu_o(\vec\Phi^{(3)}_j,\vbJ)}{\|\vec\Phi^{(3)}_j\|^2},
\end{align*}
the boundary conditions  $ \vec{\bm e}_3 \times
\vec{\bm{E}}^{o}|_{x_3 = 0} = 0$, which imply
\begin{align*}
a^{+(1)}_{j} + b^{+(1)}_{j} = 0, \\
a^{+(2)}_{j} + b^{+(2)}_{j} = 0,
\end{align*}
and the radiation conditions $ a^{-(1)}_{j} = a^{-(2)}_j = 0$ for $x_3 < -L$.

\section{Derivation of the dyadic Green's function}
\label{sect:proofLem2}
It is straightforward to check that $\mathbb{G}$ given in
\eqref{eq:TensorForm} satisfies equation (\ref{eq:Tensor2}), provided
that $\vec{G}_j$ satisfies \eqref{eq:Green2}.  To calculate
$\vec{G}_j$, we make the following observations. On $\pa\Omega$, where
$\vbn = (\bm n, 0)$,
\begin{align*}
\vbn \times \vec\Phi^{(s)}_j = -\big[(\bm
  n^\perp,0)\cdot \vec\Phi^{(s)}_j\big] \vec{\bm{e}}_3 = 0, \text{ for } s=1,2,
\text{ and } \vbn \times \vec\Phi^{(3)}_j = 0.
\end{align*}
Moreover, for a regular function $g(x_3)$ we have
\begin{align*}
\vec\nabla[\vec\nabla \cdot(g(x_3)\vec \Phi^{(1)}_j(\bm x)]  &= 0, \nonumber \\
\vec\nabla[\vec\nabla \cdot(g(x_3)\vec \Phi^{(2)}_j(\bm x)] &=
-\lambda_j g(x_3) \vec \Phi^{(2)}_j(\bm x) - \lambda_j \pa_{x_3}
g(x_3) \vec \Phi^{(3)}_j(\bm x), \\ \vec\nabla[\vec\nabla
  \cdot(g(x_3)\vec \Phi^{(3)}_j(\bm x)] &= \pa_{x_3} g(x_3) \vec
\Phi^{(2)}_j(\bm x) +\pa_{x_3}^2 g(x_3) \vec \Phi^{(3)}_j(\bm x). \nonumber
\end{align*}
These observations imply that for all $s$ and $\vx = (\bx,x_3)$, with 
$\bx \in \pa \Omega$,
\begin{align*}
\vbn(\vx) \times \vec\nabla[\vec\nabla \cdot
  \big(g(x_3)\vec\Phi^{(s)}_j(\bm x)\big)] = \vbn(\vx) \times
\big[g(x_3)\vec \Phi^{(s)}_j(\bm x)\big] = 0.
\end{align*}
This allows us to seek $ \vec{G}_j(\cdot,\vby)$ as an expansion in the
orthogonal basis $\{\vec\Phi^{(s)}_n(\bx)\}$ of eigenfunctions of the
vectorial Laplacian
\begin{equation}
\label{eq:G1}
\vec{G}_j(\vx, \vby) = \sum_{ n\in\N^2_0}\sum_{s=1}^{m_n}
g^{j(s)}_n(x_3,\vby) \vec\Phi^{(s)}_n(\bx),
\end{equation}
because each term satisfies the required boundary conditions at $\pa
\Omega$.

Substituting \eqref{eq:G1} in \eqref{eq:Green2} gives
\[
 \sum_{ n\in\N^2_0}\sum_{s=1}^{m_n}[\pa_{x_3}^2
   g^{j(s)}_n(x_3,\vby) + (k^2-\lambda_n) g^{j(s)}_n(x_3,\vby)]
 \vec\Phi^{(s)}_n(\bx) = \delta(\bx -\bm{y})\vec{\bm{e}}_j
 \delta(x_3-y_3),
\]
and using the orthogonality of the eigenfunctions we obtain the following 
ordinary differential equations
\[
\pa_{x_3}^2 g^{j(s)}_n(x_3,\vby) + (k^2-\lambda_n)
g^{j(s)}_n(x_3,\vby) = \frac{ \vec{\bm{e}}_j\cdot
  \vec\Phi^{(s)}_{n}(\bm{y})}{\|\vec\Phi^{(s)}_{n}\|^2}\delta(x_3-y_3).
\]
The solutions of these equations, which satisfy the radiation
condition at $x_3 < y_3$, are
\begin{align}
\label{eq:g}
g^j_n(x_3,\vby) = \begin{cases} (a^{(s)}_n e^{i\beta_nx_3} +
  b^{(s)}_ne^{-i\beta_nx_3})\vec{\bm{e}}_j\cdot
  \vec\Phi^{(s)}_{n}(\bm{y})/ \| \vec\Phi^{(s)}_{n}\|^2,\quad
  x_3>y_3,\\ c^{(s)}_ne^{-i\beta_nx_3} \vec{\bm{e}}_j\cdot
  \vec\Phi^{(s)}_{n}(\bm{y}) / \| \vec\Phi^{(s)}_{n}\|^2,\quad
  x_3<y_3.
\end{cases}
\end{align}
The coefficients $a^{(s)}_n, b^{(s)}_n$ and $c^{(s)}_n$ are determined
by jump conditions at $y_3$
\begin{align}
\nonumber 
g^{j(s)}_n(y_3^+,\vby) - g^{j(s)}_n(y_3^-,\vby) &= 0, \\
\label{eq:condition2}
\pa_{x_3}g^{j(s)}_n(y_3^+,\vby)- \pa_{x_3}g^{j(s)}_n(y_3^-,\vby) &=
\frac{\vec{\bm{e}}_j\cdot
  \vec\Phi^{(s)}_{n}(\bm{y})}{\|\vec\Phi^{(s)}_{n}\|^2}, 
\end{align}
and at $x_3 = 0$,
\begin{equation}
\label{eq:condition3}
\vec{\bm{e}}_3 \times (k^2+ \vec\nabla
\divv)\vec{\bm G}_j(\vx,\vby) = 0.
\end{equation}
The jump conditions (\ref{eq:condition2}) imply 
\begin{align}
\nonumber a^{(s)}_n e^{i\beta_ny_3} + b^{(s)}_n e^{-i\beta_ny_3} -
c^{(s)}_n e^{-i\beta_ny_3} =0, \\
\label{eq:eq2}
a^{(s)}_n e^{i\beta_ny_3} - b^{(s)}_n e^{-i\beta_ny_3} + c^{(s)}_n
e^{-i\beta_ny_3} = \frac{1}{i\beta_n}.
\end{align}
For the boundary condition~\eqref{eq:condition3}, we need the formulae
\begin{align*}
\vec{\bm{e}}_3 \times \vec\Phi^{(1)}_n &= \left( \begin{matrix}
  \frac{\pi n_1}{L_1} \sin \left(\frac{\pi n_1x_1}{L_1}\right) \cos
  \left(\frac{\pi n_2x_2}{L_2}\right) \\ \frac{\pi n_2}{L_2}\cos
  \left( \frac{\pi n_1x_1}{L_1}\right) \sin \left( \frac{\pi
    n_2x_2}{L_2}\right) \\ 0 \end{matrix}\right), \\ \vec{\bm{e}}_3
\times \vec\Phi^{(2)}_n &= \left( \begin{matrix} -\frac{\pi n_2}{L_2}
  \sin \left(\frac{\pi n_1x_1}{L_1}\right) \cos \left(\frac{\pi
    n_2x_2}{L_2}\right) \\ \frac{\pi n_1}{L_1}\cos \left( \frac{\pi
    n_1x_1}{L_1}\right) \sin \left( \frac{\pi n_2x_2}{L_2}\right)
  \\ 0 \end{matrix}\right), \\ \vec{\bm{e}}_3 \times \vec\Phi^{(3)}_n
&= 0,
\end{align*}
and
\begin{align*}
\vec{\bm{e}}_3 \times\vec\nabla[\vec\nabla
  \cdot(g^{j(1)}_n(x_3)\vec\Phi^{(1)}_n(\bm x)] &= 0,
\\ \vec{\bm{e}}_3 \times \vec\nabla[\vec\nabla
  \cdot(g^{j(2)}_n(x_3)\vec\Phi^{(2)}_n(\bm x)] &= -\lambda_n
g^{j(2)}_n(x_3) \left( \begin{matrix} -\frac{\pi n_2}{L_2} \sin
  \left(\frac{\pi n_1x_1}{L_1}\right) \cos \left(\frac{\pi
    n_2x_2}{L_2}\right) \\ \frac{\pi n_1}{L_1}\cos \left( \frac{\pi
    n_1x_1}{L_1}\right) \sin \left( \frac{\pi n_2x_2}{L_2}\right)
  \\ 0 \end{matrix}\right), \\ \vec{\bm{e}}_3
\times\vec\nabla[\vec\nabla \cdot(g^{j(3)}_n(x_3)\vec\Phi^{(3)}_j(\bm
  x)] &= \pa_{x_3} g^{j(3)}_n(x_3) \left( \begin{matrix} -\frac{\pi
    n_2}{L_2} \sin \left(\frac{\pi n_1x_1}{L_1}\right) \cos
  \left(\frac{\pi n_2x_2}{L_2}\right) \\ \frac{\pi n_1}{L_1}\cos
  \left( \frac{\pi n_1x_1}{L_1}\right) \sin \left( \frac{\pi
    n_2x_2}{L_2}\right) \\ 0 \end{matrix}\right).
\end{align*}
Substituting in \eqref{eq:condition3} we get 
\begin{align*}
g_n^{j(1)}(0,\vby) = 0 \text{ and } (k^2-\lambda_n)g^{j(2)}_n(0,\vby)
+\pa_{x_3}g^{j(3)}_n(0,\vby) = 0,
\end{align*}
or, equivalently,
\begin{align}
\label{eq:eq3p}
a_n^{(1)} + b_n^{(1)} = 0,
\end{align}
and
\begin{align} (k^2-\lambda_n)
\big(a_n^{(2)} + b_n^{(2)} \big) \frac{\vec{\bm e}_j \cdot \vec
  \Phi_j^{(2)}(\vby)}{\|\vec \Phi_j^{(2)}\|^2} + i \beta_n
\big(a_n^{(3)} - b_n^{(3)} \big) \frac{\vec{\bm e}_j \cdot \vec
  \Phi_j^{(3)}(\vby)}{\|\vec \Phi_j^{(3)}\|^2} &= 0.\label{eq:eq3}
\end{align}

We now have a linear system of eight equations (\ref{eq:eq2}),
(\ref{eq:eq3p}) and (\ref{eq:eq3}) for the nine unknowns $a_n^{(s)}$,
$b_n^{(s)}$ and $c_n^{(s)}$. The system is underdetermined, so
$\vec{\bm G}_j$ is not uniquely defined. However, $\mathbb{G}(\cdot,
\vby)$ given by~\eqref{eq:TensorForm} is unique, because a
straightforward computation shows that the coefficients with $s = 2$
or $3$ appear only in the combinations
\[
(b_n^{(3)} + i\beta_nb_n^{(2)}) \left( \vec\Phi^{(2)}_n +
\frac{i\lambda_n}{\beta_n} \vec\Phi^{(3)}_n(\bm x) \right)
e^{-i\beta_nx_3} \] and
\[
(a_n^{(3)} - i\beta_n a_n^{(2)}) \left(\vec\Phi^{(2)}_n -
\frac{i\lambda_n}{\beta_n} \vec\Phi^{(3)}_n(\bm x) \right)
e^{i\beta_nx_3}.
\]
Thus, we can calculate the most convenient solution of the
underdetermined system (\ref{eq:eq2}), (\ref{eq:eq3p}) and
(\ref{eq:eq3}), corresponding to $a_n^{(3)} = b_n^{(3)}$. This gives
$\pa_{x_3} g_n^{j(3)}(0) = 0.$ The expression of $\vec{\bm G}_j$ in 
Lemma \ref{lem.2} follows.

\bibliographystyle{siam}
\bibliography{ip-biblio}
\end{document}